\documentclass[12pt,a4paper]{article}

\usepackage[utf8]{inputenc}
\usepackage[english]{babel}
\usepackage{amsmath}
\usepackage{amsfonts}
\usepackage{amssymb}
\usepackage{graphicx}

\usepackage{amsthm}
\usepackage{mathrsfs}
\usepackage{proof}

\usepackage[left=3cm,right=2cm,top=3cm,bottom=2cm]{geometry}

\theoremstyle{theorem}
\newtheorem{theorem}{Theorem}[section]
\newtheorem{lemma}[theorem]{Lemma}
\newtheorem{proposition}[theorem]{Proposition}
\newtheorem{corollary}[theorem]{Corollary}

\theoremstyle{definition}
\newtheorem{definition}[theorem]{Definition}
\newtheorem{remark}[theorem]{Remark}
\newtheorem{Not}[theorem]{\textbf{Notation}}

\newcommand{\fol}{\textbf{CFOL}}
\newcommand{\mbc}{\textbf{mbC}}
\newcommand{\qmbc}{\textbf{QmbC}}
\newcommand{\lfis}{{\bf LFI}s}
\newcommand{\lfi}{{\bf LFI}}

\newcommand{\cpl}{\text{\bf CPL}}
\newcommand{\dacdot}{{\bf J3}}
\newcommand{\lfium}{{\bf LFI1}}
\newcommand{\lfiuml}{{\bf LFI1'}}
\newcommand{\qlfium}{{\bf QLFI1}}
\newcommand{\cons}{\ensuremath{{\circ}}}
\newcommand{\defin}{\ensuremath{~\stackrel{\text{{\tiny def }}}{=}~}}
\newcommand{\mA}{\ensuremath{\mathcal{A}}}
\newcommand{\mB}{\ensuremath{\mathcal{B}}}
\newcommand{\matM}{\ensuremath{\mathcal{M}}}
\newcommand{\qmbceq}{\ensuremath{\textbf{QmbC}^\approx}}

\newcommand{\termvalue}[1] {\lbrack\!\lbrack #1 \rbrack\!\rbrack}

\newcommand{\aptz}{\ensuremath{\mA_{LFI1}}}

\newcommand{\A}{\ensuremath{\mathcal{A}}}
\newcommand{\B}{\ensuremath{\mathcal{B}}}
\newcommand{\tA}{\ensuremath{T_\mathcal{A}}}
\newcommand{\tmA}{\ensuremath{\mathcal{T}_\mathcal{A}}}
\newcommand{\matA}{\ensuremath{\mathcal{MT}_\mathcal{A}}}

\newcommand{\axdneg}{\textbf{dneg}}
\newcommand{\axnegou}{\ensuremath{\textbf{neg}\lor}}
\newcommand{\axnege}{\ensuremath{\textbf{neg}\land}}
\newcommand{\axnegimp}{\ensuremath{\textbf{neg}\imp}}
\newcommand{\axci}{\textbf{ci}}

\newcommand{\wneg}{\ensuremath{\lnot}}
\newcommand{\sneg}{\ensuremath{{\sim}}}
\newcommand{\imp}{\rightarrow}

\newcommand{\valpred}{\textit{vPred}}
\newcommand{\valou}{\textit{vOr}}
\newcommand{\vale}{\textit{vAnd}}
\newcommand{\valimp}{\textit{vImp}}
\newcommand{\valnot}{\textit{vNeg}}
\newcommand{\valbola}{\textit{vCon}}
\newcommand{\valvar}{\textit{vVar}}
\newcommand{\valex}{\textit{vEx}}
\newcommand{\valuni}{\textit{vUni}}
\newcommand{\valsubs}{\textit{vSubs}}
\newcommand{\valeq}{\textit{vEq}}

\newcommand{\sent}{$Sen(\Theta)$}
\newcommand{\tert}{$Ter(\Theta)$}
\newcommand{\fort}{$For(\Theta)$}
\newcommand{\ctert}{$CTer(\Theta)$}
\newcommand{\sena}{\ensuremath{Sen(\Theta_U)}}

\newcommand{\fora}{$For(\Theta_U)$}

\begin{document}

\title{First-order swap structures semantics for some Logics of Formal Inconsistency}

\author{Marcelo E.~Coniglio$^{1}$, Aldo Figallo-Orellano$^{2}$ and Ana C. Golzio$^{3}$\\ [2mm] %
{\small $^1$Institute of Philosophy and the Humanities (IFCH) and}\\
{\small Centre for Logic, Epistemology and The History of Science (CLE),}\\
{\small University of Campinas (UNICAMP), Campinas, SP, Brazil.}\\
{\small E-mail:~{coniglio@unicamp.br}}\\[1mm]
{\small $^2$Departmento de Matem\'atica, Universidad Nacional del Sur (UNS),}\\
{\small  Bahia Blanca, Argentina and}\\
{\small Centre for Logic, Epistemology and The History of Science (CLE),}\\
{\small University of Campinas (UNICAMP), Campinas, SP, Brazil.}\\
{\small E-mail:~{aldofigallo@gmail.com}}\\[1mm]
{\small $^3$S\~ao Paulo State University (Unesp), Marilia Campus, Brazil.}\\
{\small E-mail:~{anaclaudiagolzio@yahoo.com.br}}}

\date{}

\maketitle

\begin{abstract}
The logics of formal inconsistency (\lfis, for short) are  paraconsistent logics (that is, logics containing contradictory but  non-trivial theories) having  a consistency connective which allows to recover the {\em ex falso quodlibet} principle in a controlled way. The aim of this paper is considering a novel semantical approach to first-order \lfis\ based on Tarskian structures defined over swap structures, a special class of  multialgebras. The proposed semantical framework generalizes previous aproaches to quantified \lfis\ presented in the literature. The case of \qmbc, the simpler quantified \lfi\ expanding classical logic, will be analyzed in detail. An axiomatic extension of \qmbc\ called $\qlfium_\circ$ is also studied, which is equivalent to the quantified version of da Costa and D'Ottaviano 3-valued logic \dacdot. The semantical structures for this logic turn out to be Tarkian structures based on twist structures. The expansion of \qmbc\ and $\qlfium_\circ$ with a standard equality predicate is also considered.  
\end{abstract}

\noindent {\em Keywords}: First-order logics; Logics of formal inconsistency; Paraconsistent logics; Swap structures; Non-deterministic matrices; Twist structures

\section{Introduction}
\label{intro}

A logic is said to be {\em paraconsistent} if it contains in its language a negation and it has a contradictory theory  (with respect to such negation) which is  non-trivial. Such negation is called a {\em paraconsistent} or {\em non-explosive} negation. This is why paraconsistent logics are said to be {\em tolerant to contradictions}.
The class of paraconsistent logics known as {\em logics of formal inconsistency} (\lfis, for short) was introduced by W. Carnielli and J. Marcos in~\cite{CM02}. In its simplest form, they have a non-explosive negation $\neg$, as well as  a (primitive or derived) {\em consistency connective} $\circ$ which allows to recover the explosion law in a controlled way. 

Defining interesting and ellucidative semantics for paraconsistent logics is a  challenging task for paraconsistentists, taking into account that, in general, these logics are not algebraizable by means of the standard techniques. Several kinds of semantics of non-deterministic character were proposed for these logics, in particular for \lfis. Among them, the non-deterministic matrices (or Nmatrices), introduced by A. Avron and I. Lev in~\cite{avr:lev:01} (see also~\cite{avr:lev:05}) constitutes an interesting and useful semantical framework for dealing with such logics.

The situation is even more delicate in the case of first-order paraconsistent logics. The aim of this paper is considering a novel semantical approach for first-order \lfis\ based on Tarskian structures defined over a special class of  multialgebras called {\em swap structures}, which were introduced by Carnielli and Coniglio in~\cite{CC16}. The proposed semantical framework for quantified \lfis\ generalizes previous aproaches in the literature.

In order to understand what lies behind the non-deterministic semantical framework  for first-order \lfis\ proposed here, let us recall briefly the standard lattice-theoretic algebraic approach to first-order classical logic (\fol, in short). To simplify the exposition, let us consider only sentences instead of formulas with free variables. Let us begin by recalling  that the standard semantics for \fol\ is given by  Tarskian first-order structures. Given such a structure $\mathsf{A}$, in order to evaluate sentences of the given language, it is used the standard two-valued logical matrix for propositional classical logic \cpl\ based on  the two-element Boolean algebra $\mathcal{A}_2$ with domain $A_2=\{0,1\}$, expanded by quantification operators $\tilde{Q}: (\mathcal{P}(A_2)-\{\emptyset\}) \to A_2$ for $Q\in \{\forall, \exists\}$ given, respectively, by
$$\tilde{\forall}(X) =  \bigwedge X \ \mbox{ and } \ \tilde{\exists}(X) =  \bigvee  X.   $$
It is known that this class of semantical structures can be enlarged by considering structures $\mathfrak{A} = \langle U, I_{\mathfrak{A}} \rangle$ defined over a complete Boolean algebra $\mathcal{A}$ instead of the two-element Boolean algebra $\mathcal{A}_2$ (see, for instance, \cite[Supplement--Section~8]{rasiowa}. This means that $I_\mathfrak{A}(P)$ is now a mapping   $I_\mathfrak{A}(P):U^n \to A$ for any $n$-ary predicate symbol $P$, where $A$ is the domain of $\mathcal{A}$. The interpretation of sentences in \fol\ is then given in the context $(\mathfrak{A},\mathcal{M}_\A)$, where $\mathcal{M}_\A=\langle \A,\{1\}\rangle$ constitutes a logical matrix for \cpl\ in which $1$ is the only designated value, and where the quantifiers are interpreted {\em mutatis mutandis} as in the case of $\mathcal{A}_2$. This kind of algebraic approach to quantified logics was firstly proposed by Mostowski in~\cite{mostowski}, and later on generalized by Henkin in~\cite{henkin} and Rasiova and Sikorski (see~\cite{rasiowa-sikorski,rasiowa}). This approach was recently extended by Cintula and Noguera in~\cite{Ci:Nog:15} to first-order logics based on algebraizable logics in the sense of Blok and Pigozzi.

What is proposed here for \qmbc\  is a generalization of the algebraic approach to \fol\ mentioned above, in which the logical matrix $\mathcal{M}_\A$ is replaced by a non-deterministic matrix $\mathcal{M}(\B)$ for \mbc, where \B\ is a swap structure  (a multialgebra of a special kind) for \mbc\ which is induced by a complete Boolean algebra \A. That is, sentences of \qmbc\ will be interpreted in contexts $(\mathfrak{A},\mathcal{M}(\B))$ for any swap structure \B\ for \mbc. The analogy with the semantics of  \fol\ is more accurate than it first appears: while \fol\ is fully characterized by the class of contexts $(\mathfrak{A},\mathcal{M}_2)$ such that $\mathcal{M}_2=\langle \mathcal{A}_2, \{1\}\rangle$ (which is equivalent, up to presentation, to the class of standard Tarskian structures with the usual semantics), \qmbc\ can be characterized by contexts $(\mathfrak{A},\mathcal{M}_5)$ defined over the 5-element Nmatrix $\mathcal{M}_5$, which is the one induced by the greatest swap structure defined over  $\mathcal{A}_2$, as we shall see in Section~\ref{compM5}. This is why semantical contexts for \qmbc\ of the form $(\mathfrak{A},\mathcal{M}_5)$ are considered to be `classical', in analogy to \fol. It is worth noting that the semantics for \qmbc\ given by the `classical' contexts coincide, up to notational aspects, with the non-deterministic semantics proposed by Avron and Zamansky in~\cite{avr:zam:07}.

An interesting feature of the mutialgebraic semantics sudied here is that, by adding axioms to a given \lfi, conditions on the multioperations, and even on the domain of the swap structures, naturally arise. In particular, consider the 3-valued $\lfium_\cons$, which is equivalent (up to language) to several well-known 3-valued logics such as da Costa-D'Ottaviano logic \dacdot. As shown in~\cite{CFG18}, the swap structures for $\lfium_\cons$, which is an axiomatic extension of \mbc, are deterministic, and as such they become  twist structures, that is, ordinary agebras of a certain kind. From this, the first-order swap structures for $\qlfium_\circ$, the quantified extension of $\lfium_\cons$, become first-order twist structures. As we shall see, when the `classical' structures (that is, the twist structures induced by the Boolean algebra $\mathcal{A}_2$) are considered, the corresponding first-order structures are defined over the characteristic 3-valued logical matrix for $\lfium_\cons$, hence this semantics is equivalent (up to presentation) with the early 3-valued model theory for quantified \dacdot\ introduced in~\cite{dot:82}, \cite{dot:85a}, \cite{dot:85b} and~\cite{dot:87}.

This paper is organized as follows: in the first sections, the swap structures semantics for \mbc\ will be recalled.   In Section~\ref{swapfol} a semantics based on swap stuctures for \qmbc\ will be introduced, proving in the following sections the corresponding soundness and completeness theorems. As we shall see, the swap structures semantics generalizes the interpretation semantics for \qmbc\ given in~\cite{CC16}, as well as the Nmatrix semantics proposed in~\cite{avr:zam:07}.
The extension of \qmbc\ by adding a standard equality predicate will be  analyzed in Sections~\ref{secEq} and~\ref{compM5eq}. The analysis of $\qlfium_\circ$,  whose semantics is based on twist structures, will be done in  Sections ~\ref{Qtwist} to~\ref{comp3val}. Some conclusions are given in the last section.

\section{The logic \mbc} 

In this section, the notion of logics of formal inconsistency will be recalled, and the basic \lfi\ called \mbc\ will be briefly described.
Let $\Sigma'$ be a propositional signature, and assume a denumerable set $\mathcal{V}=\{p_1,p_2,\ldots\}$ of propositional variables. The propositional language generated by $\Sigma'$ from $\mathcal{V}$ will be denoted by ${\cal L}_{\Sigma'}$

\begin{definition}
Let  ${\bf L}=\langle  \Sigma',\vdash \rangle$ be a Tarskian, finitary and structural logic defined over a propositional signature $\Sigma'$, which contains a
negation $\neg$, and let  $\circ$ be a (primitive or defined) unary connective. Then,
${\bf L}$ is said to be  a {\em logic of formal inconsistency} with respect to $\neg$ and $\circ$ if the following holds:\footnote{The original definition of \lfis\ allows to consider a set $\bigcirc(p)$ of formulas depending on a single propositional variable $p$, instead of a single formula $\circ p$. See~\cite{CM02,CCM,CC16}.}\\
(i) \ $\varphi,\neg\varphi\nvdash\psi$ for some $\varphi$ and $\psi$;\\
(ii) \ there are two formulas $\alpha$ and $\beta$ such that\\
\indent (ii.a) \ $\circ\alpha,\alpha \nvdash \beta$;\\
\indent (ii.b) \ $\circ\alpha, \neg \alpha \nvdash \beta$; \\
(iii) \ $\circ\varphi,\varphi,\neg\varphi\vdash \psi$ for every $\varphi$ and $\psi$. 
\end{definition}

\noindent Condition~(iii) states that {\em ex falso quodlibet} is controllably recovered in \lfis\  by assuming that the contradictory formula $\varphi$ is consistent, i.e., $\cons\varphi$.

\begin{definition}  (\cite[Definition 2.1.12]{CC16}) Let $\Sigma = \{\wedge,\vee,\to,\neg, \circ\}$ be a propositional signature. The calculus \mbc\ over the propositional language ${\cal L}_{\Sigma}$ is defined by means of the following Hilbert calculus:\\[2mm]
{\bf Axiom schemas:}\\

$\begin{array}{ll}
{\bf (A1)} & \alpha \to(\beta \to \alpha)\\[2mm]
{\bf (A2)} & (\alpha \to (\beta \to \gamma)) \to ((\alpha \to \beta) \to (\alpha \to \gamma ))\\[2mm]
{\bf (A3)} & \alpha \to(\beta \to (\alpha \wedge \beta))\\[2mm]
{\bf (A4)} & (\alpha \wedge \beta)\to \alpha\\[2mm]
{\bf (A5)} & (\alpha \wedge \beta)\to \beta\\[2mm]
{\bf (A6)} & \alpha \to( \alpha \vee \beta)\\[2mm]
{\bf (A7)} & \beta \to( \alpha \vee \beta)\\[2mm]
{\bf (A8)} & (\alpha \to \gamma)\to (( \beta \to \gamma)\to ((\alpha \vee \beta)\to \gamma))\\[2mm]
{\bf (A9)} & \alpha \vee (\alpha \to \beta)\\[2mm]
{\bf (A10)} & \alpha \vee \neg \alpha\\[2mm]
{\bf (A11)} & \circ \alpha \to (\alpha \to (\neg \alpha \to \beta))\\[4mm]
\end{array} $

{\bf Inference rule:}\\

$\begin{array}{ll}
{\bf (MP)} & \dfrac{\alpha ~~~\alpha \to\beta}{\beta}\\
\end{array} $
\end{definition}

\noindent
Observe that \mbc\ is obtained from positive classical logic by adding axioms {\bf (A10)} and {\bf (A11)} governing the new connectives $\neg$ (paraconsistent negation) and $\circ$ (consistency operator). It is easy to see that $\mbc$ is an \lfi. Indeed, it is the least \lfi\ which contains propositional classical logic \cpl\ (see Remark~\ref{univ-Henk}).

\section{Swap structures for \mbc} \label{swapmbc}

It is well-known that \mbc, as well as several axiomatic extensions of it, are neither agebraizable (see~\cite[Section~3.12]{CM02}), nor characterizable by a single finite logical matrix (see for instance~\cite[Theorems~121 and~125]{CCM}). In this section a non-deterministic semantics for \mbc\ based on multialgebras called {\em swap structures}, introduced in~\cite[Chapter~6]{CC16}, will be briefly recalled.
  
\begin{definition} Let $\Omega$ be a propositional signature. A {\em multialgebra} (or {\em hyperalgebra}) over $\Omega$ is a pair $\mathcal{A}= \langle A,\sigma\rangle$ such that $A$ is a nonempty set (the {\em universe} or {\em support} of $\mathcal{A}$) and $\sigma$  is a mapping assigning to each $n$-ary connective $c$, a function (called {\em multioperation} or {\em hyperoperation})  $c^\mathcal{A}:A^n \to(\mathcal{P}(A)-\{\emptyset\})$.  In particular, $\emptyset \neq c^\mathcal{A} \subseteq A$ if $c$ is a constant symbol. 
\end{definition}

\begin{definition} Let $\Omega$ be a propositional signature. A \textit{non-deterministic matrix (or Nmatrix)} is a pair $ \mathcal{M} = \langle \mathcal{A}, D \rangle $ such that $\mathcal{A}= \langle A,\sigma\rangle$ is a multialgebra over $\Omega$ with support $A$, and $D$ is a subset of $A$. The elements in $D$ are called \textit{designated} elements. 
\end{definition}

\begin{Not} \label{Pi1} {\em Let $\mathcal{A}$ be a Boolean algebra with domain $A$. If $x \in A \times A \times A$ then  $(x)_i$ (or simply $x_i$) will denote the $ith$-projection of $x$, that is, $\pi_i(x)$, where $\pi_i$ is the $ith$-canonical projection for $i=1,2,3$.}
\end{Not}

\begin{definition} (\cite[Definition 6.4.1]{CC16}) \label{defswap}
Let ${\cal A} = (A, \wedge, \vee, \to, 0,1)$ be a Boolean algebra, and 
$B_{\cal A} = \{x \in A \times A \times A \ : \ x_1 \vee x_2 = 1 \ \mbox{ and } \ x_1 \wedge x_2 \wedge x_3 = 0\}$.
A {\em swap structure for \mbc\ over ${\cal A}$} is any multialgebra $ {\cal B} = (B , \tilde{\wedge}, \tilde{\vee}, \tilde{\to}, \tilde{\neg}, \tilde{\circ})$ over $\Sigma$ 
such that $B \subseteq B_{\cal A}$ and where the multioperations satisfy the following, for every $x$ and
$y$ in $B$:\\[1mm]
(i) $\emptyset\not= x\tilde{\#}y \subseteq \{ z\in B \ : \ z_1=x_1\# y_1\}$, for each $\# \in \{\wedge,\vee,\to\}$;\\
(ii) $\emptyset\not= \tilde{\neg} x \subseteq \{ z\in B \ : \ z_1=x_2\}$;\\
(iii) $\emptyset\not=\tilde{\circ} x \subseteq \{ z\in B \ : \ z_1=x_3\}$.
\end{definition}

\begin{definition} \label{deffull}
The  {\em full  swap structure  for \mbc\ over  \mA},  denoted by ${\cal B}_{\cal A}$, is the unique swap structure for {\bf mbC} over \mA\ with domain $B_{\cal A}$, in which `$\subseteq$' is replaced by `$=$' in items (i)-(iii) of Definition~\ref{defswap}.
\end{definition}

\noindent Observe that  ${\cal B}_{\cal A}$ is the greatest swap structure for \mbc\ over $\mathcal{A}$ (see~\cite{CFG18}).
The elements of a given swap structure are called  {\em snapshots}. This terminology is inspired by its use in computer science to refer to states. Accordingly, a triple $(a, b, c)$ of a swap structure $\cal B$ keeps track simultaneously of the value $a$ of a given formula $\varphi$, a possible value $b$ for $\neg \varphi$, and a possible value $c$ for $\circ \varphi$. 

Given that any swap structure is a multialgebra, the consequence relation over swap structures will be defined by means of non-deterministic matrices, in analogy with the corresponding notion for twist structures.

\begin{definition} (\cite[Definition 6.4.3]{CC16}) \label{defNmat}  Given a Boolean algebra $\cal A$ and a swap structure $\cal B$ for {\bf mbC} over {\cal A} with domain $B$, let $D_B = \{x\in B \ : \ x_1 = 1	\}$ be the set of {\em designated} elements. 
The non-deterministic matrix associated to $\cal B$ is ${\cal M}({\cal B}) =({\cal B}, D_B)$ (or simply ${\cal M}({\cal B}) =(\mathcal{B},D)$).
The Nmatrix associated to ${\cal B}_{\cal A}$ will be denoted by ${\cal M}_{\cal A}$. 
The class of all the Nmatrices defined by swap structures for {\bf mbC} will be denoted by $\mathbb{M}_{\mbc}$, that is:
$$\mathbb{M}_{\mbc} = \{ {\cal M}({\cal B}) \ : \ {\cal B}\,\, \mbox{ is a swap structure for \mbc\ over ${\cal A}$, for some ${\cal A}$}\}.$$ 
\end{definition}

\begin{definition} \label{defval} Let ${\cal M}({\cal B}) \in \mathbb{M}_{\mbc}$. A {\em valuation} over ${\cal M}({\cal B})$ is a function $v:{\cal L}_{\Sigma}\to |\mathcal{B}|$ such that, for every $\varphi,\psi \in {\cal L}_{\Sigma}$:\\
(i) \  $v(\varphi \# \psi) \in v(\varphi) \tilde{\#} v(\psi)$, for every $\#\in \{\wedge,\vee, \to\}$;\\
(ii) \ $v(\neg\varphi) \in \tilde{\neg} v(\varphi)$;\\
(iii) \ $v(\circ\varphi) \in \tilde{\circ} v(\varphi)$.
\end{definition}

\begin{definition}  \label{semNmat} 
Let ${\cal M}({\cal B}) \in \mathbb{M}_{\mbc}$, and let $\Gamma \cup \{\varphi\} \subseteq {\cal L}_{\Sigma}$. We say that $\varphi$ is a {\em consequence} of $\Gamma$  in ${\cal M}({\cal B})$, denoted by $\Gamma\models_{{\cal M}({\cal B})} \varphi$, if the following holds: for every valuation  $v$ over ${\cal M}({\cal B})$, if $v[\Gamma] \subseteq D$ then $v(\varphi) \in D$. In particular, $\varphi$ is {\em valid} in ${\cal M}({\cal B})$, denoted by $\models_{{\cal M}({\cal B})} \varphi$, if $v(\varphi) \in D$ for every valuation  $v$ over ${\cal M}({\cal B})$. The {\em swap consequence relation $\models_{\mathbb{M}_{\mbc}}$ for \mbc} is given by: $\Gamma \models_{\mathbb{M}_{\mbc}} \varphi$ whenever $\Gamma\models_{{\cal M}({\cal B})} \varphi$ for every ${\cal M}({\cal B}) \in \mathbb{M}_{\mbc}$.
\end{definition}

\begin{theorem} {\em (Adequacy of \mbc\ w.r.t. swap structures, \cite[Theorem~6.4.8]{CC16} and~\cite[Theorem~7.1]{CFG18})} \label{adeq-mbC}
For every $\Gamma \cup\{ \varphi\} \subseteq {\cal L}_{\Sigma}$:  $\Gamma \vdash_\mbc \varphi$ \ iff \ $\Gamma \models_{\mathbb{M}_{\mbc}} \varphi$.
\end{theorem}

\section{The 5-valued characteristic Nmatrix $\mathcal{M}_5$ for \mbc} \label{M5}

It is illustrative to compare Theorem~\ref{adeq-mbC} with the adequacy of classical propositional logic \cpl\ w.r.t. Boolean algebras semantics. As it is well known, it is enough to consider just one Boolean algebra to semantically characterize \cpl, namely the two-element Boolean algebra $\mA_2$ with domain $A_2=\{0,1\}$ and the associated logical matrix with $1$ as designated value. In the case of swap structures semantics, it is enough to consider the Nmatrix $\mathcal{M}_5=\mathcal{M}\big(\mathcal{B}_{\mA_2}\big)$ induced by the full swap structure $\mathcal{B}_{\mA_2}$ defined over  $\mA_2$. The Nmatrix $\mathcal{M}_5$ was originally introduced by A. Avron in  \cite{avr:05} to semantically characterize \mbc. Observe that  $B_{\mA_2} = \big\{T, \, t, \, t_0, \, F, \, f_0\big\}$ where $T=(1,0,1)$, $t=(1,1,0)$, $t_0=(1,0,0)$, $F=  (0,1,1)$, and $f_0=(0,1,0)$. 
The set D of designated elements of $\mathcal{M}_5$ is $\textrm{D}=\{T, \, t, \, t_0\}$, while $\textrm{ND}=\big\{F, \, f_0\big\}$ is the set of non-designated truth-values. 
The multioperations proposed by Avron over the set $B_{\mA_2}$ coincide  with  the corresponding ones for $\mathcal{B}_{\mA_2}$, and so his 5-valued Nmatrix coincides with $\mathcal{M}\big(\mathcal{B}_{\mA_2}\big)$. Observe that the swap structure of $\matM_5$ is defined as follows:

{\small
\begin{center}
\begin{tabular}{|c|c|c|c|c|c|}
\hline
 $\wedge^{\matM_5}$ & $T$   & $t$ & $t_0$ & $F$ & $f_0$ \\
 \hline \hline
    $T$    & D  & D & D & ND & ND   \\ \hline
     $t$    & D  & D & D & ND & ND  \\ \hline
     $t_0$    & D  & D & D & ND & ND  \\ \hline
     $F$    & ND  & ND & ND & ND & ND  \\ \hline
     $f_0$    & ND  & ND & ND & ND & ND  \\ \hline
\end{tabular}
\hspace{0.5cm}
\begin{tabular}{|c|c|c|c|c|c|}
\hline
 $\vee^{\matM_5}$ & $T$   & $t$ & $t_0$ & $F$ & $f_0$ \\
 \hline \hline
    $T$    & D  & D & D & D & D   \\ \hline
     $t$    & D  & D & D & D & D  \\ \hline
     $t_0$    & D  & D & D & D & D  \\ \hline
     $F$    & D  & D & D & ND & ND  \\ \hline
     $f_0$    & D  & D & D & ND & ND  \\ \hline
\end{tabular}
\end{center}

\begin{center}
\begin{tabular}{|c|c|c|c|c|c|}
\hline
 $\to^{\matM_5}$ & $T$   & $t$ & $t_0$ & $F$ & $f_0$ \\
 \hline \hline
    $T$    & D  & D & D & ND & ND   \\ \hline
     $t$    & D  & D & D & ND & ND  \\ \hline
     $t_0$    & D  & D & D & ND & ND  \\ \hline
     $F$    & D  & D & D & D & D  \\ \hline
     $f_0$    & D  & D & D & D & D  \\ \hline
\end{tabular}
\hspace{0.5cm}
\begin{tabular}{|c||c|} \hline
$\quad$ & $\neg^{\matM_5}$ \\
 \hline \hline
    $T$   & ND    \\ \hline
     $t$   & D    \\ \hline
     $t_0$   & ND    \\ \hline
     $F$   & D    \\ \hline
     $f_0$   & D    \\ \hline
\end{tabular}
\hspace{0.5cm}
\begin{tabular}{|c||c|}
\hline
 $\quad$ & $\circ^{\matM_5}$ \\
 \hline \hline
    $T$   & D    \\ \hline
     $t$   & ND    \\ \hline
     $t_0$   & ND    \\ \hline
     $F$   & D    \\ \hline
     $f_0$   & ND    \\ \hline
\end{tabular}
\end{center}
}

\

\begin{theorem} {\em (Adequacy of \mbc\ w.r.t. $\matM_5$, \cite[Theorem~3.6]{avr:05})} \label{comp-mbc-full}
For every set of formulas $\Gamma \cup \{\varphi\} \subseteq {\cal L}_{\Sigma}$: 
$\Gamma \vdash_\mbc \varphi$ \ iff \  $\Gamma\models_{\mathcal{M}_5} \varphi$.
\end{theorem}

\noindent
A new proof of this result was obtained in~\cite[Corollary~6.4.10]{CC16}, by using bivaluations for \mbc\  in connection with the Nmatrix $\mathcal{M}\big(\mathcal{B}_{\mA_2}\big)$.

\begin{definition} (\cite[Definition~54]{CCM}) \label{bivalold}
A function $\rho:{\cal L}_{\Sigma}\to \big\{0,1\big\}$ is a {\em bivaluation for \mbc}
if it  satisfies the   following clauses:\\[1mm]
{\bf (\vale)} \ $\rho(\alpha \land \beta) = 1$  \ iff \  $\rho(\alpha) = \rho(\beta) = 1$ \\[1mm]
{\bf (\valou)} \ $\rho(\alpha \lor \beta) = 1$ \  iff \ $\rho(\alpha) = 1$ \ or \
$\rho(\beta) = 1$ \\[1mm]
{\bf (\valimp)} \ $\rho(\alpha \to \beta) = 1$ \  iff \  $\rho(\alpha) = 0$ \ or \
$\rho(\beta) = 1$ \\[1mm]
{\bf (\valnot)} \ $\rho(\lnot \alpha)=0$ \ implies  \ $\rho(\alpha)=1$ \\[1mm]
{\bf (\valbola)} \ $\rho(\cons \alpha) = 1$ \ implies \ $\rho(\alpha)=0$ \ or \ $\rho(\lnot \alpha)=0$. \\[1mm]
The consequence relation of \mbc\ w.r.t. bivaluations, denoted by  $\models_\mbc^2$, is given by: $\Gamma \models_\mbc^2 \varphi$ \ iff \ $\rho(\varphi)=1$ for every bivaluation for \mbc\ such that $\rho[\Gamma] \subseteq \{1\}$.
\end{definition}

\begin{theorem} {\em (Adequacy of \mbc\ w.r.t. bivaluations, \cite[Theorems~56 and~61]{CCM})}  \label{comp-bival-mbC}
For every set of formulas $\Gamma \cup \{\varphi\} \subseteq {\cal L}_{\Sigma}$: 
$\Gamma \vdash_\mbc \varphi$ \ iff \  $\Gamma\models_{\mbc}^2 \varphi$.
\end{theorem}

\begin{theorem} {\em (\cite[Theorem~6.4.9]{CC16})}  \label{val-bival-mbC}
Any  bivaluation $\rho$ for \mbc\ induces a valuation $v^\rho$ over the Nmatrix $\matM_5$ given by $v^\rho(\alpha) \defin (\rho(\alpha),\rho(\neg\alpha),\rho(\cons\alpha))$ such that, for every formula $\alpha$: $\rho(\alpha)=1$ iff $v^\rho(\alpha) \in {\rm D}$.
\end{theorem}

\noindent From the previous result, Theorem~\ref{comp-mbc-full} follows easily (see~\cite[Corollary~6.4.10]{CC16}). 
The characteristic Nmatrix $\matM_5$ of  \mbc\ can be considered as the `classical' model of it, since it is based on the `classical'  Boolean algebra $\mA_2$. In Section~\ref{compM5} it will be shown that \qmbc, the first-order version of \mbc, can be characterized by first-order structures defined over $\matM_5$. These structures can be considered as `classical'  in this sense.

\section{The logic \qmbc}
In this section the first-order logic \qmbc, introduced in~\cite{car:etal:14} (see also~\cite{CC16}) as an extension  of \mbc\  to first-order languages, will be briefly recalled. In Section~\ref{swapfol} a new semantics of first-order swap structures for \qmbc\ will be defined.
 
\begin{definition} \label{fosig}
Assume the propositional signature $\Sigma = \{\wedge,\vee,\to,\neg, \circ\}$  for \mbc, as well as the symbols $\forall$ (universal quantifier) and $\exists$ (existential quantifier), with the punctuation marks (commas and parentesis). Let $Var=\{v_1,v_2,\ldots\}$ be a denumerable set of individual variables. A first-order signature $\Theta$ for \qmbc\ is composed by the following elements: 
\begin{itemize}
  \item[-] a set $\mathcal{C}$ of individual constants; 
  \item[-] for each $n\geq 1$, a set $\mathcal{F}_n$ of function symbols of arity $n$,
  \item[-] for each $n\geq 1$, a set $\mathcal{P}_n$ of predicate symbols of arity $n$.\footnote{It will be assumed, as usual, that $\Theta$ has at least one predicate symbol, in order to have a non-empty set of formulas. For instance, it could be assumed from the beginning an equality predicate (see Section~\ref{secEq}).}
\end{itemize}
\end{definition}

\begin{Not} {\em Let $\Theta$ be a first-order signature for \qmbc. The sets of terms and formulas generated by $\Theta$ from $Var$ we will denoted by \tert\ and \fort, respectively. The set of sentences (formulas without free variables) and the set of closed terms (terms without variables) over $\Theta$ are denoted by \sent\ and \ctert, respectively. Given a formula $\varphi$, the formula obtained from $\varphi$ by substituting every free occurrence of a variable $x$ by a term $t$ will be denoted by $\varphi[x/t]$.}
\end{Not}

\noindent
The notions of {\em subformula}, {\em scope} of an occurrence of a quantifier in a formula, {\em free} and {\em bound} occurrences of a variable in a formula, and of  {\em term free for a variable in a formula}, are the usual ones (see, for instance, \cite{mendelson}).

\begin{definition} (\cite[Definition 7.1.4]{CC16}) Let $\varphi$ and $\psi$ be formulas. If $\varphi$ can be obtained from $\psi$ by means of addition or deletion of void quantifiers, or by renaming bound variables (keeping the same free variables in the same places), we say that $\varphi$ and $\psi$ are {\em variant} of each other.
\end{definition}

\begin{definition} (\cite[Definition 7.1.5]{CC16}) \label{defqmbc}
Let $\Theta$ be a first-order signature. The logic \qmbc\ is obtained from the Hilbert calculus \mbc\ extended by the following axioms and rules:\\[2mm]
{\bf Axiom schemas:}\\

$\begin{array}{ll}
{\bf (Ax12)} & \varphi[x/t]\to\exists x\varphi, \ \mbox{ if $t$ is a term free for $x$ in $\varphi$}\\[3mm]
{\bf (Ax13)} & \forall x\varphi\to\varphi[x/t], \ \mbox{ if $t$ is a term free for $x$ in $\varphi$}\\[3mm]
{\bf (Ax14)} & \alpha\to\beta, \ \mbox{ whenever $\alpha$ is a variant of $\beta$}
\end{array}$\\[4mm]

{\bf Inference rules:}\\

$\begin{array}{ll}
{\bf (\exists\mbox{\bf -In})} & \dfrac{\varphi\to\psi}{\exists x\varphi\to\psi}, \ \mbox{  where $x$ does not occur free in $\psi$}\\[4mm]
{\bf (\forall\mbox{\bf -In})} & \dfrac{\varphi\to\psi}{\varphi\to\forall x\psi}, \ \mbox{ where $x$ does not occur free in $\varphi$}
\end{array}$
\end{definition}

\begin{definition} If $\Gamma \cup \{\varphi \} \subseteq For(\Theta)$, then $\Gamma \vdash_{\qmbc} \varphi $ will denote that there exists a derivation in \qmbc\ of $\varphi$ from $\Gamma$.
\end{definition}

\noindent In~\cite{CC16} it was proved that the logic \qmbc\ enjoys the Deduction meta-theorem (DMT), as usually presented in first-order logics:

\begin{theorem}[Deduction Meta-Theorem (DMT) for \qmbc]\label{teoded:teo}
Suppose that there exists in \qmbc\ a derivation of $\psi$ from $\Gamma \cup\{\varphi\}$, such that no application of the rules {\em($\exists${\bf -In})} and {\em($\forall${\bf -In})} have, as their quantified variables, free variables of $\varphi$ (in particular, this holds when $\varphi$ is a sentence). Then $\Gamma \vdash_{\qmbc} \varphi \to \psi$.
\end{theorem}

\section{First-Order Swap Structures} \label{swapfol}

The traditional approach to first-order structures based on algebraic structures (see for instance~\cite{mostowski,henkin,rasiowa-sikorski}) will be adapted to swap structures semantics. 
Thus,  from now on the  Boolean algebras to be considered are assumed to be  complete.\footnote{Instead of this we could consider arbitrary Boolean algebras, thus obtaining partial models in which  some quantified formulas has no denotation because of the lack of infima and/or suprema of some subsets. Since every Boolean algebra can be completed (see Section~\ref{complete}), we  decide to restrict the semantic to complete Boolean algebras. An interesting discussion concerning this topic can be found in~\cite[footnote 3]{Ci:Nog:15}.} 

\begin{definition} \label{stru} Let ${\cal M}({\cal B})=(\mathcal{B},D)$ be a non-deterministic matrix defined by a swap structure $\mathcal{B}$ for {\bf mbC} over a complete Boolean algebra \mA, and let $\Theta$ be a first-order signature (see Definition~\ref{fosig}). A  (first-order) {\em structure} over ${\cal M}({\cal B})$ and $\Theta$ is pair $\mathfrak{A} = \langle U, I_{\mathfrak{A}} \rangle$ such that $U$ is a nonempty set (the domain of the structure) and $I_{\mathfrak{A}}$ is an interpretation function which assigns:
\begin{itemize}
\item[-] to each individual constant $c \in \mathcal{C}$, an element $I_\mathfrak{A}(c)$ of $U$; 
\item[-] to each function symbol $f$ of arity $n$, a function $I_\mathfrak{A}(f): U^n \to
U$; 
\item[-] to each predicate symbol $P$ of arity $n$, a function  $I_\mathfrak{A}(P): U^n \to |{\cal B}|$.
\end{itemize}
\end{definition}

\noindent
From now on, the expressions $c^\mathfrak{A}$, $f^\mathfrak{A}$ and $P^\mathfrak{A}$ will be used instead of $I_\mathfrak{A}(c)$, $I_\mathfrak{A}(f)$ and $I_\mathfrak{A}(P)$, for an individual constant symbol c, a function symbol $f$ and a predicate symbol $P$, respectively.

\begin{definition} \label{assign} Let $\mathfrak{A}$ be a structure over ${\cal M}({\cal B})$ and $\Theta$. A function $\mu: Var \to U$ is called an {\em assignment} over  $\mathfrak{A}$. 
\end{definition}

\begin{definition}~\label{term} 
Let $\mathfrak{A}$ be a structure over ${\cal M}({\cal B})$ and $\Theta$, and let $\mu: Var \to U$ be an assignment. For each term $t$, we define $\termvalue{t}^\mathfrak{A}_\mu$ in $U$ such that:
\begin{itemize}
\item[-] $\termvalue{c}^\mathfrak{A}_\mu = c^\mathfrak{A}$ if $c$ is an individual constant;
\item[-] $\termvalue{x}^\mathfrak{A}_\mu = \mu(x)$ if $x$ is a variable;
\item[-] $\termvalue{f(t_1,\ldots,t_n)}^\mathfrak{A}_\mu = f^\mathfrak{A}(\termvalue{t_1}^\mathfrak{A}_\mu,\ldots,\termvalue{t_n}^\mathfrak{A}_\mu)$ if $f$ is a function symbol of arity $n$ and $t_1,\ldots,t_n$ are terms.
\end{itemize}
\end{definition}

\begin{definition} \label{diaglan} 
Consider a structure $\mathfrak{A}$ over ${\cal M}({\cal B})$ and $\Theta$. The {\em diagram language} of $\mathfrak{A}$ is the set of formulas $For(\Theta_U)$, where  $\Theta_U$ is the signature obtained from $\Theta$ by adding a  new individual constant $\bar{a}$ for each element $a$ of the domain $U$ of $\mathfrak{A}$. 
\end{definition}

\begin{definition} \label{extA}
The structure $\widehat{\mathfrak{A}} = \langle U, I_{\widehat{\mathfrak{A}}} \rangle$ over $\Theta_U$  is the structure $\mathfrak{A}$ over $\Theta$ extended by $I_{\widehat{\mathfrak{A}}}(\bar{a})=a$ for every $a \in U$. 
\end{definition}

\noindent Observe that $s^{\widehat{\mathfrak{A}}} = s^\mathfrak{A}$ if $s$ is a symbol (individual constant, function symbol or predicate symbol) of $\Theta$.

\begin{Not} \label{SentA}
{\em For any formula $\varphi$,  $FV(\varphi)$ will denote the set of free variables of $\varphi$.
The set of (closed) sentences (formulas without free variables) of the diagram language of $\mathfrak{A}$ will be denoted by \sena, and the set of terms and of closed terms over $\Theta_U$ will be denoted by $Ter(\Theta_U)$ and $CTer(\Theta_U)$, respectively. }
\end{Not}

\begin{remark} \label{closedterm} Clearly, if $t$ is a closed term then the value of $\termvalue{t}^\mathfrak{A}_\mu$ does not depend on the  assignment $\mu$, that is:  $\termvalue{t}^\mathfrak{A}_\mu = \termvalue{t}^\mathfrak{A}_{\mu'}$, for every assignments $\mu$ and $\mu'$. Thus, if  $t$ is a closed term we can write $\termvalue{t}^\mathfrak{A}$ instead of $\termvalue{t}^\mathfrak{A}_\mu$, for any   assignment $\mu$.
\end{remark}

\begin{definition}~\label{mu}
Let $\widehat{\mathfrak{A}} = \langle U, I_{\widehat{\mathfrak{A}}} \rangle$ be as above. Any assignment $\mu$ over $\widehat{\mathfrak{A}}$ induces 
a function $\widehat{\mu}: (Ter(\Theta_U) \cup For(\Theta_U)) \to (Ter(\Theta_U) \cup For(\Theta_U))$ given by $\widehat{\mu}(s) = s[x_1/\overline{\mu(x_1)}, \ldots, x_n/\overline{\mu(x_n)}]$, if either $s \in Ter(\Theta_U)$ such that  $Var(s)\subseteq \{x_1, \ldots, x_n \}$ or $s \in For(\Theta_U)$ such that $FV(s)\subseteq \{x_1, \ldots, x_n \}$.
\end{definition}

\noindent
It is worth observing that $\widehat{\mu}(\varphi) \in Sen(\Theta_U)$ if $\varphi \in For(\Theta_U)$, and $\widehat{\mu}(t) \in CTer(\Theta_U)$ if $t \in Ter(\Theta_U)$.

The next step is to define the notion of interpretation (or denotation) of a formula $\varphi \in For(\Theta_U)$ in a given (extended) structure $\widehat{\mathfrak{A}}$ and assignment $\mu$, which could be denoted by $\termvalue{\varphi}^{\mathfrak{A}}_\mu$ (being coherent with the previous notation).
Is exactly at this point when non-determinism enters. Observe that, in the traditional (truth-functional or algebraic)  first-order semantical approach, any structure and assignment induce together a (unique) denotation for any formula. In the present framework, this is also true  for atomic formulas, since predicates are interpreted by means of functions, and taking into account that the denotation of any term is uniquely determined given a structure and an assignment. However, the denotation of complex formulas is possibly non-deterministic (i.e., ambiguous), given that it involves logical symbols (connectives and quantifiers) to be evaluated over a non-deterministic matrix. As happens with the propositional case, we are not interested in assigning sets of truth-values to single formulas: instead of this, valuations ({\em legal valuations}, in Avron and Lev's terminology) are used in order to {\em choose}, in a coherent way, a single truth-value for any formula.\footnote{As we shall see in Section~\ref{Qtwist},  in the case of the first-order logic $\lfium_\cons$, which is an axiomatic extension of \qmbc\ based on a truth-functional 3-valued logic, any structure and assignment will determine a unique truth-value for any single formula. In this case, it will not necessary to consider valuations.}  The definition of (legal) valuations over first-order swap structures involves an additional technical complication with respect to the propositional case: the validity of the Substitution Lemma -- a crucial result which allows to substitute a universally quantified variable by any term free for such variable in a given formula -- is far from being true in our non-deterministic environment.  Indeed, this technical result is trivially true for first-order logics in which the semantics is  obtained by algebraic manipulations over the interpretation of the subformulas of the formula being interpreted. Since in \qmbc\ it is necessary to introduce  the valuations as intermediaries between the formulas and the multioperators of the swap structures, such valuations must satisfy additional requirements in order to guarantee the validity of the Substitution Lemma (namely, clause~(vi) in Definition~\ref{val} below).
To summarize,  in order to interpret formulas in the present non-deterministic framework, it is necesary a structure, an assignment, and a (first-order) valuation over the underlying swap structure, which will be called a {\em \qmbc-valuation}.

Given an assignment $\mu$ over a structure $\mathfrak{A}$, a variable $x$ and $a \in U$, the assignment $\mu^{x}_{a}$ over $\mathfrak{A}$ is given by $\mu^{x}_{a}(y) = a$, if $y = x$, and   $\mu^{x}_{a}(y) =  \mu(y)$ otherwise. Thus, the previous considerations lead us to the following notion:

\begin{definition} (\qmbc-valuations)~\label{val}
Let ${\cal M}({\cal B}) =({\cal B}, D)$ be a non-de\-ter\-min\-is\-tic matrix defined by a swap structure ${\cal B}$ for \mbc, and let $\mathfrak{A} $ be a structure over $\Theta$ and ${\cal M}({\cal B})$. A function $v:$ \sena $\to |{\cal B}|$ is a {\em valuation} for \qmbc\ (or a \qmbc-{\em valuation}) over $\mathfrak{A} $ and  ${\cal M}({\cal B})$,  if it satisfies the following clauses:\\
($i$) $v(P(t_1,\ldots,t_n)) = P^{\mathfrak{A}}(\termvalue{t_1}^{\widehat{\mathfrak{A}}},\ldots,\termvalue{t_n}^{\widehat{\mathfrak{A}}}) $, if $P(t_1,\ldots,t_n)$ is atomic;\footnote{For the notation used here, recall Remark~\ref{closedterm}.} \\
($ii$) $v(\#\varphi) \in \tilde{\#} v(\varphi)$, for every $\#\in \{\neg, \circ\}$;\\
($iii$)  $v(\varphi \# \psi) \in v(\varphi) \tilde{\#} v(\psi)$, for every $\#\in \{\wedge,\vee, \to\}$;\\
($iv$) $v(\forall x\varphi) \in \{ z \in |{\cal B}| \ : \ z_1 = \bigwedge\{\pi_1(v(\varphi[x/\bar{a}])) \ : \ a \in U\}\}$;\\
($v$) $v(\exists x\varphi) \in \{ z \in |{\cal B}| \ : \ z_1 = \bigvee\{\pi_1(v(\varphi[x/\bar{a}])) \ : \ a \in U\}\}$;\\
($vi$)] Let $t$ be free for $z$ in $\varphi$ and $\psi$, $\mu$ an assignment and  $b = \termvalue{t}^{\widehat{\mathfrak{A}}}_{\mu}$. Then:\\
($vi.1$) If $v(\widehat{\mu}(\varphi[z/t])) = v(\widehat{\mu}(\varphi[z/\bar{b}]))$, then  $v(\widehat{\mu}(\#\varphi[z/t])) = v(\widehat{\mu}(\#\varphi[z/\bar{b}]))$, for every $\#\in \{\neg, \circ\}$;\\
($vi.2$) If $v(\widehat{\mu}(\varphi[z/t])) = v(\widehat{\mu}(\varphi[z/\bar{b}]))$ and $v(\widehat{\mu}(\psi[z/t])) = v(\widehat{\mu}(\psi[z/\bar{b}]))$, then \\$v(\widehat{\mu}(\varphi\#\psi[z/t]))$ $= v(\widehat{\mu}(\varphi\#\psi[z/\bar{b}]))$, for every $\#\in \{\wedge,\vee, \to\}$;\\
($vi.3$) Let $x$ be such that $x \neq z$ and $x$ does not occur in $t$. If $v(\widehat{\mu^{x}_{a}}(\varphi[z/t])) = v(\widehat{\mu^{x}_{a}}(\varphi[z/\bar{b}]))$, for every $a \in U$, then  $v(\widehat{\mu}((Qx\varphi)[z/t])) = v(\widehat{\mu}((Qx\varphi)[z/\bar{b}]))$, for every $Q\in \{\forall, \exists\}$;\\
($vii$) If $\varphi$ and $\varphi'$ are variant, then $v(\varphi) = v(\varphi')$.
\end{definition}

Observe that clause (i) in the previous definition is the only one that uses the information of the structure $\mathfrak{A}$, and it allows to interpret the atomic formulas. In order to obtain a single denotation for a complex formula, the valuation is used to choose (coherently) a denotation for the formula from the denotation of its components. Clause~(vi) guarantees the validity of the Substitution Lemma, a crucial step for obtaining the soundness of the proposed semantics.

\begin{definition} \label{vmu} Let  $v$ be a \qmbc-valuation over $\mathfrak{A}$ and ${\cal M}({\cal B})$. Given an assignment $\mu$ over $\mathfrak{A} $, we define $v_{\mu}:$ \fora $\to |{\cal B}|$ as $v_{\mu}(\varphi) = v(\widehat{\mu}(\varphi))$.
\end{definition}

\begin{definition} \label{semcons} Let  $\mathfrak{A} $ be a structure over $\Theta$ and ${\cal M}({\cal B})$. If $\Gamma \cup \{\varphi\} \subseteq$ \fora, $\varphi$ is said to be  a {\em semantical consequence of $\Gamma$ over $(\mathfrak{A}, {\cal M}({\cal B}))$}, denoted by $\Gamma\models_{(\mathfrak{A}, {\cal M}({\cal B}))}\varphi$, if the following holds: for every \qmbc-valuation $v$ over $(\mathfrak{A}, {\cal M}({\cal B}))$, if $v_{\mu}(\gamma)\in D$,  for every formula $\gamma \in \Gamma$ and every assignment $\mu$, then $v_{\mu}(\varphi)\in D$,  for every assignment $\mu$.
\end{definition}

\begin{definition} (Semantical consequence relation in \qmbc\ w.r.t. swap structures) \label{consrel} If $\Gamma \cup \{\varphi\} \subseteq For(\Theta)$,  $\varphi$ is said to be  a {\em semantical consequence of $\Gamma$ in \qmbc\ w.r.t. first-order swap structures}, denoted by $\Gamma\models_{\qmbc}\varphi$, if $\Gamma\models_{(\mathfrak{A}, {\cal M}({\cal B}))}\varphi$ for every $(\mathfrak{A}, {\cal M}({\cal B}))$.
\end{definition}

\noindent
As mentioned in the Introduction, the semantical contexts $(\mathfrak{A}, {\cal M}({\cal B}))$ for \qmbc\  generalize the semantical contexts $(\mathfrak{A},\mathcal{M}_\A)$ for first-order classical logic \fol, where $\mathcal{M}_\A=\langle \A,\{1\}\rangle$. The latter, by its turn, generalize the class of standard Tarskian structures for \fol\ with the usual semantics, by taking the two-element Boolean algebra  $\mA_2$.

\section{Soundness of \qmbc\ w.r.t. swap structures}

In this section the  soundness of \qmbc\ w.r.t. first-order swap structures semantics for \qmbc\ will be proved. As mentioned in the previous section, a key result for  proving soundness is the Substitution Lemma, which can be proved easily by induction on the complexity of $\varphi$.

\begin{theorem} {\em (Substitution Lemma)} \label{substlem}
Let  $v$ be a \qmbc-valuation over $\mathfrak{A}$ and ${\cal M}({\cal B})$ and let $\mu$ be an assignment. If $t$ is a term free for $z$ in $\varphi$ and $b = \termvalue{t}^{\widehat{\mathfrak{A}}}_{\mu}$, then $v_{\mu}(\varphi[z/t]) = v_{\mu}(\varphi[z/\bar{b}])$.
\end{theorem}

\noindent A useful property of the semantics of the universal quantifier can be obtained now. The easy proof is ommited.

\begin{proposition}~\label{propquant}  
Let  $v$ be a \qmbc-valuation over $\mathfrak{A}$ and ${\cal M}({\cal B})$, and let $\varphi$ be a formula such that $FV(\varphi) \subseteq \{x_1, \ldots, x_n\}$. Then, $v(\forall x_1 \ldots\forall x_n \varphi) \in D$ iff $v_\mu(\varphi) \in D$, for every $\mu$.
\end{proposition}
 
\noindent 
If $\alpha$ and $\beta$ are formulas in $For(\Theta)$ then $\alpha \leftrightarrow \beta$ will denote the formula $(\alpha \to \beta) \wedge (\beta \to \alpha)$ in $For(\Theta)$.

\begin{proposition}~\label{paxioms} 
\ \\
(i) \ $\alpha, \, \alpha \to \beta \models_\qmbc \beta$; \\
(ii) \ $\alpha \to \beta \models_\qmbc \exists x \alpha \to \beta$, if $x \notin FV(\beta)$;\\ 
(iii) \ $\alpha \to \beta \models_\qmbc \alpha \to \forall x \beta$, if $x \notin FV(\alpha)$;\\
(iv) \ $\models_\qmbc \forall x \alpha \to \alpha[x/t]$, if $t$ is a term free for $x$ in $\alpha$;\\
(v) \ $\models_\qmbc \alpha[x/t] \to \exists x \alpha$, if $t$ is a term free for $x$ in $\alpha$;\\
(vi) \ $\models_\qmbc \alpha \leftrightarrow \alpha' $, if $\alpha$ and $\alpha'$ are variant.
\end{proposition}

\begin{proof} (i): It is obvious from the definitions.\\ 
(ii): 
Let $v$  be a \qmbc-valuation over $(\mathfrak{A}, {\cal M}({\cal B}))$ such that $v_{\mu}(\alpha \to \beta) \in D$, for every assignment $\mu$. Hence $(v_{\mu}(\alpha))_1 \leq (v_{\mu}(\beta))_1$, for every $\mu$. From this, for every $a \in U$:  $(v_{\mu^{x}_{a}}(\alpha))_1 \leq (v_{\mu^{x}_{a}}(\beta))_1=(v_{\mu}(\beta))_1$, since  $x \notin FV(\beta)$. But then: $(v_{\mu}( \exists x \alpha))_1 = \bigvee_{a \in U} (v_{\mu}(\alpha[x/a]))_1 = \bigvee_{a \in U} (v_{\mu^{x}_{a}}(\alpha))_1 \leq (v_{\mu}(\beta))_1$. Hence, $v_{\mu}( \exists x \alpha \to \beta) \in D$,  for every $\mu$. This shows that $\alpha \to \beta \models_\qmbc \exists x \alpha \to \beta$. Item (iii) is proved analogously.\\[1mm]
(iv): 
Assume that  $t$ is a term free for $x$ in $\alpha$. Let $v$  be a \qmbc-valuation over $(\mathfrak{A}, {\cal M}({\cal B}))$ and let $\mu$ be an assignment. If $b = \termvalue{t}^{\mathfrak{A}}_{\mu}$, then, by Theorem~\ref{substlem}, $v_\mu(\alpha[x/t]) = v_\mu(\alpha[x/\bar{b}])$. Then, $(v_{\mu}(\forall x \alpha))_1 = \bigwedge_{a \in U} (v_{\mu}(\alpha[x/\bar{a}))_1 \leq (v_\mu(\alpha[x/\bar{b}]))_1 = (v_\mu(\alpha[x/t]))_1$. Hence,  $v_{\mu}(\forall x \alpha \to \alpha[x/t]) \in D$.  Item (v) is proved analogously.\\[1mm]
(vi): Let $v$  be a \qmbc-valuation and let $\mu$ be an assignment. If $\alpha$ and $\alpha'$ are variant, so are  $\widehat{\mu}(\alpha)$ and $\widehat{\mu}(\alpha')$. By  Definition~\ref{val}(vii), $v_{\mu}(\alpha \leftrightarrow \alpha') \in D$.
\end{proof}

\begin{corollary}~\label{psound}
Let $v$ be a \qmbc-valuation over $\mathfrak{A}$ and ${\cal M}({\cal B})$. Then:\\
(1)  If $\alpha$ is an instance of a \qmbc\ axiom schema then $v_{\mu}(\alpha) \in D$, for every assignment $\mu$.\\
(2) If $\alpha$ and $\beta$ are formulas such that $v_{\mu}(\alpha) \in D$ and $v_{\mu}(\alpha \to \beta) \in D$ for every assignment $\mu$, then $v_{\mu}(\beta) \in D$ for every assignment $\mu$.\\
(3) If $\alpha$ and $\beta$ are formulas such that $v_{\mu}(\alpha \to \beta) \in D$ for every assignment $\mu$, and if $x$ does not occur free in $\beta$, then $v_{\mu}(\exists x \alpha\to \beta) \in D$ for every $\mu$.\\
(4) If $\alpha$ and $\beta$ are formulas such that $v_{\mu}(\alpha \to \beta) \in D$, for every  $\mu$, and if $x$ does not occur free in $\alpha$, then $v_{\mu}(\alpha\to \forall x\beta) \in D$ for every  $\mu$.
\end{corollary}

\begin{proof}
Item (1) follows by Theorem 2.2.2 in~\cite{CC16} and by Proposition~\ref{paxioms}(iv)-(vi). The rest of the proof follows by  Proposition~\ref{paxioms}(i)-(iii).
\end{proof}

\noindent
From this corollary it follows easily:

\begin{theorem} {\em (Soundness of \qmbc\ w.r.t. first-order swap structures)} \label{sound-Qmbc-swap}
For every set $\Gamma \cup\{ \varphi\}  \subseteq For(\Theta)$:  if $\Gamma \vdash_\qmbc \varphi$, then  $\Gamma \models_\qmbc \varphi$.
\end{theorem}

\section{Completeness of \qmbc\ w.r.t. swap structures} \label{complete}

In this section the  completeness of \qmbc\ w.r.t. first-order swap structures semantics for \qmbc\ will be obtained. In order to do this, some definitions and results given in Sections 7.5.1 and  7.5.2 of~\cite{CC16} for proving  the completeness theorem for \qmbc\ w.r.t. interpretations will be adapted. In addition, the technique for proving the completeness of \mbc\ w.r.t. swap structures presented in~\cite[Theorem~7.1]{CFG18} will be also used. The first step is considering a notion of $C$-Henkin theory a bit stronger than the one proposed in~\cite[Definition 7.5.1]{CC16}.  

\begin{definition}\label{Henkin}
Consider a theory $\Delta \subseteq$ \sent\ and a nonempty set $C$ of constants of the  signature $\Theta$. Then, $\Delta$ is called a
$C$-\emph{Henkin theory} in \qmbc\ if it satisfies the following: for every formula $\varphi$   with (at most) a free variable $x$, there exists a constant
$c$ in $C$ such that $\Delta \vdash_\qmbc \exists x  \varphi \to \varphi[x/c]$.
\end{definition}

\begin{remark} \label{univ-Henk}
Recall by \cite[Section 2.4]{CC16} that $\bot_\beta \defin \beta \wedge(\neg\beta \wedge \circ\beta)$ is a bottom in \mbc, hence $\sneg_\beta\alpha \defin \alpha \to \bot_\beta$ is a classical negation in \mbc. This construction does not depend on $\beta$ (up to logical equivalence), hence we will write $\sneg\alpha$ instead of $\sneg_\beta\alpha$. This can be also done in \qmbc.
By \cite[Proposition~7.2.2]{CC16} it follows that $\exists x  \sneg\varphi \to \sneg\varphi[x/c] \vdash_\qmbc \varphi[x/c] \to \forall x\varphi$.
Thus, if $\Delta$ is  a $C$-Henkin theory in \qmbc\ and $\varphi$ is a formula with (at most) a free variable $x$  then there is a constant
$c$ in $C$ such that   $\Delta \vdash_\qmbc \varphi[x/c] \to \forall x  \varphi$.
\end{remark}

\begin{definition} \label{Qmbc-C} Let $\Theta_{C}$  be the signature obtained from $\Theta$ by adding a set $C$ of new individual constants. The consequence relation $\vdash^{C}_{\qmbc}$ is the consequence relation of \qmbc\ over the signature $\Theta_{C}$.
\end{definition}

\noindent
Recall that, given a Tarskian and finitary logic  ${\bf L}=\langle  For,\vdash \rangle$ (where $For$ is the set of formulas of {\bf L}), and given a set $\Gamma \cup \{\varphi\} \subseteq For$, the set $\Gamma$ is said to be {\em maximally non-trivial with respect to $\varphi$ in {\bf L}} if the following holds:~(i)~$\Gamma \nvdash \varphi$, and~(ii)~$\Gamma,\psi \vdash\varphi$ for every $\psi \notin \Gamma$.

\begin{proposition} {\em (\cite[Corollary 7.5.4]{CC16})} \label{saturated}
Let  $\Gamma \cup \{\varphi\} \subseteq$ \sent\ such that $\Gamma \nvdash_{\qmbc} \varphi$. Then, there exists a set of sentences $\Delta \subseteq$ \sent\ which is maximally non-trivial with respect to $\varphi$ in \qmbc\ (by restricting $\vdash_{\qmbc}$ to sentences) and such that $\Gamma \subseteq \Delta$.
\end{proposition}

\begin{definition} Let $\Delta \subseteq Sen(\Theta)$ be non-trivial in \qmbc, that is: there is some sentence $\varphi$ in $Sen(\Theta)$ such that $\Delta \nvdash_{\qmbc} \varphi$. Let ${\equiv_{\Delta}} \subseteq$ \sent$^{2}$ be the relation in \sent\ defined as follows:  $\alpha \equiv_{\Delta} \beta$ iff $\Delta \vdash_{\qmbc} \alpha \leftrightarrow\beta$. 
\end{definition}

\noindent
By adapting the proof of~\cite[Theorem~7.1]{CFG18} it follows that $\equiv_{\Delta}$ is an equivalence relation. Moreover, in the quotient set  $A_{\Delta} \defin$ \sent$/_{\equiv_{\Delta}}$ it is possible to define binary operators $\bar\wedge$, $\bar\vee$, $\bar\to$ as follows: $[\alpha]_\Delta \bar\# [\beta]_\Delta \defin [\alpha \# \beta]_\Delta$ for any $\# \in \{\wedge, \vee, \to\}$, where $[\alpha]_\Delta$ denotes the equivalence class of formula $\alpha$ w.r.t. $\equiv_{\Delta}$. Using the axioms of \qmbc\ coming from  \mbc\ it follows:

\begin{proposition} \label{Adelta}
The structure $\mathcal{A}_{\Delta} \defin\langle A_{\Delta}, \bar\wedge, \bar\vee, \bar\to, 0_{\Delta},1_{\Delta}\rangle$ is a Boolean algebra with $0_{\Delta} \defin [\varphi \wedge (\neg \varphi \wedge \circ \varphi)]_{\Delta}$ and  $1_{\Delta} \defin [\varphi \vee \neg \varphi]_{\Delta}$, for any sentence~$\varphi$.
\end{proposition}

\noindent
In order to construct the canonical model for \qmbc\ w.r.t. $\Delta$, the Boolean algebra $\mathcal{A}_{\Delta}$ needs to be completed. Recall (see, for instance, \cite[Chapter~25]{gi:ha:09}) that a Boolean algebra \mB\ is a {\em completion} of a Boolean algebra \mA\ if:~(1) \mB\ is complete, and~(2) \mB\ includes \mA\ as a dense subalgebra (that is: every element in $B$ is the supremum, in \mB, of some subset of $A$). As a consequence of the definition, it follows that \mB\ preserves all the existing infima and suprema in \mA. In formal terms: there exists a monomorphism of Boolean algebras (therefore an injective mapping) $\ast:\mA \to \mB$ such that $\ast(\bigvee_\mA X)= \bigvee_\mB \ast[X]$  for every $X\subseteq A$ such that the supremum $\bigvee_\mA X$ exists, where $\ast[X]=\{\ast(a) \ : \ a \in X\}$. Analogously, $\ast(\bigwedge_\mA X)= \bigwedge_\mB \ast[X]$  for every $X\subseteq A$ such that the infimum $\bigwedge_\mA X$ exists. By the (independent) results of MacNeille and Tarski, it is known that every Boolean algebra has a completion; moreover, the completion is unique up to isomorphisms. Thus, let $C\mA_\Delta$ be the completion of $\mA_\Delta$ and let  $\ast:\mA_\Delta \to C\mA_\Delta$ be the associated monomorphism.

\begin{definition}\label{nma}
 Let $C\mA_\Delta$ be the complete Boolean algebra defined as above. The full swap structure for \mbc\ over $C\mA_\Delta$ (recall Definition~\ref{deffull}) will be denoted by  $\mathcal{B}_{\Delta}$. The associated Nmatrix (recall Definition~\ref{defNmat}) will be denoted by $\mathcal{M}(\mathcal{B}_{\Delta}) \defin (\mathcal{B}_{\Delta}, D_{\Delta})$. 
\end{definition}

\noindent
Notice that $(\ast([\alpha]_{\Delta}), \ast([\beta]_{\Delta}), \ast([\gamma]_{\Delta})) \in D_{\Delta} \ \mbox{ iff } \ \Delta \vdash_{\qmbc} \alpha$.

\begin{definition} (Canonical Structure) \label{str}
Let $\Theta$ be a signature with some individual constant. Let $\Delta \subseteq Sen(\Theta)$ be non-trivial in \qmbc, let $\mathcal{M}(\mathcal{B}_{\Delta})$ be as in Definition~\ref{nma}, and let $U = $~\ctert. The {\em canonical structure induced by $\Delta$} is the  structure  $\mathfrak{A}_\Delta = \langle U, I_{\mathfrak{A}_\Delta} \rangle$ over $\mathcal{M}(\mathcal{B}_{\Delta})$ and $\Theta$ such that:
\begin{itemize}
\item[-] $c^{\mathfrak{A}_\Delta} = c $, for each individual constant $c$;
\item[-] $f^{\mathfrak{A}_\Delta}: U^n \to U$ is such that $f^{\mathfrak{A}_\Delta} (t_1, \ldots, t_n) = f(t_1, \ldots, t_n)$, for each function symbol $f$ of arity $n$;
\item[-] $P^{\mathfrak{A}_\Delta}(t_1, \ldots, t_n) = (\ast([\varphi]_{\Delta}), \ast([\neg \varphi]_{\Delta}), \ast([\circ \varphi]_{\Delta})) $ with $\varphi = P(t_1, \ldots, t_n)$, for each predicate symbol $P$ of arity $n$.
\end{itemize}
\end{definition}

\noindent Notice that $[\varphi]_{\Delta} \bar\vee [\neg \varphi]_{\Delta} = [\varphi \vee \neg \varphi]_{\Delta}= 1$ and $[\varphi]_{\Delta} \bar\wedge [\neg \varphi]_{\Delta} \bar\wedge [\cons \varphi]_{\Delta} = [\varphi \wedge \neg \varphi \wedge \cons\varphi]_{\Delta}= 0$. Thus $\ast([\varphi]_{\Delta}) \vee \ast([\neg \varphi]_{\Delta}) = \ast([\varphi]_{\Delta} \bar\vee [\neg \varphi]_{\Delta}) = 1$, and  $\ast([\varphi]_{\Delta}) \wedge \ast([\neg \varphi]_{\Delta}) \wedge \ast([\cons\varphi]_{\Delta})= \ast([\varphi]_{\Delta} \bar\wedge [\neg \varphi]_{\Delta} \bar\wedge [\cons \varphi]_{\Delta}) = 0$. Hence $P^{\mathfrak{A}_\Delta}(t_1, \ldots, t_n) \in |\mathcal{B}_{\Delta}|$ and so $\mathfrak{A}_\Delta $ is indeed a structure over $\mathcal{M}(\mathcal{B}_{\Delta})$ and $\Theta$.

\begin{definition} \label{transla}
Let $(\cdot)^\triangleright:(Ter(\Theta_U) \cup For(\Theta_U)) \to (Ter(\Theta) \cup For(\Theta))$ be the mapping such that $\left(\,s\,\right)^\triangleright$ is the expression obtained from $s$ by substituting every occurrence of a constant $\bar{t}$ by the term  $t$ itself, for $t \in CTer(\Theta)$.
\end{definition}

\noindent For instance, $\left(\,P(f(\bar{c},x)) \wedge Q\big(\,\overline{f(c,x)},z\,\big)\,\right)^\triangleright = P(f(c,x)) \wedge Q(f(c,x),z)$.

\begin{definition} (Canonical valuation) \label{valca}
Let $\Delta \subseteq$ \sent\ be a set of sentences over a signature $\Theta$ such  that $\Delta$ is a $C$-Henkin theory in \qmbc\ for a nonempty set $C$ of individual constants of $\Theta$, and $\Delta$ is maximally non-trivial with respect to $\varphi$ in \qmbc, for some sentence $\varphi$. The {\em  canonical \qmbc-valuation induced by $\Delta$ over $\mathfrak{A}_{\Delta}$ and $\mathcal{M}(\mathcal{B}_{\Delta})$} is the function $v_{\Delta}: \sena \to |\mathcal{B}_{\Delta}|$ such that $v_{\Delta}(\psi) = (\ast([(\psi)^\triangleright]_{\Delta}), \ast([\neg(\psi)^\triangleright]_{\Delta}), \ast([\circ(\psi)^\triangleright]_{\Delta}))$, for every sentence $\psi$ over $\Theta_U$.
\end{definition}

\noindent
Notice that $v_{\Delta}(\psi) \in D_{\Delta} \mbox{ iff } \Delta \vdash_{\qmbc} (\psi)^\triangleright$.

\begin{lemma} \label{quantOK}
Let $\Delta \subseteq$ \sent\ be as in Definition~\ref{valca}. Then, for every formula $\psi(x)$ in which $x$ is the unique variable (possibly) occurring free, it holds:\\[1mm]
(1) $[\forall x \psi]_\Delta = \bigwedge_{\mA_\Delta} \{ [\psi[x/t]]_\Delta \ :  \ t \in CTer(\Theta)\}$, where $\bigwedge_{\mA_\Delta}$ denotes  an existing infimum in the Boolean algebra $\mathcal{A}_{\Delta}$;\\[1mm]
(2) $[\exists x \psi]_\Delta = \bigvee_{\mA_\Delta} \{ [\psi[x/t]]_\Delta \ :  \ t \in CTer(\Theta)\}$, where $\bigvee_{\mA_\Delta}$ denotes  an existing supremum in the Boolean algebra $\mathcal{A}_{\Delta}$.
\end{lemma}
\begin{proof} \ \\
(2)
Observe that, in $\mA_\Delta$, $[\alpha]_\Delta \leq [\beta]_\Delta$ iff $\Delta \vdash_{\qmbc} \alpha \rightarrow\beta$. Let $\alpha$ be a formula in which $x$ is the unique variable (possibly) occurring free. Then $[\alpha[x/t]]_\Delta \leq [\exists x \alpha]_\Delta$ for every $t \in  CTer(\Theta)$, by {\bf (Ax12)}. Let $\beta$ be a sentence such that $[\alpha[x/t]]_\Delta \leq [\beta]_\Delta$ for every $t \in  CTer(\Theta)$. That is, $\Delta \vdash_\qmbc \alpha[x/t] \to \beta$ for every $t \in  CTer(\Theta)$. Since $\Delta$ is a $C$-Henkin theory, there is a constant $c$ of $\Theta$ such that $\Delta \vdash_\qmbc \exists x \alpha  \to \alpha[x/c]$. By hypothesis, $\Delta \vdash_\qmbc \alpha[x/c] \to \beta$ and so $\Delta \vdash_\qmbc \exists x \alpha \to \beta$. That is, $[\exists x \alpha]_\Delta \leq  [\beta]_\Delta$. This shows that $[\exists x \alpha]_\Delta = \bigvee_{\mA_\Delta} \{ [\alpha[x/t]]_\Delta \ :  \ t \in CTer(\Theta)\}$.  \\[2mm]
Item (1)  is proved analogously, but now by using Remark~\ref{univ-Henk}. 
\end{proof}

\begin{theorem} \label{canval}
The canonical \qmbc-valuation $v_{\Delta}$ is a \qmbc-valuation over $\mathfrak{A}_{\Delta}$ and $\mathcal{M}(\mathcal{B}_{\Delta})$.
\end{theorem}

\begin{proof}
Let us see that $v_{\Delta}$ satisfies all the requirements of Definition~\ref{val}.\\[1mm]
(i) If $\varphi$ is an  atomic formula $P(t_1,\ldots,t_n)$ then:\\
$v_{\Delta}(\varphi) =  (\ast([(\varphi)^{\triangleright}]_{\Delta}), \ast([\neg (\varphi)^{\triangleright}]_{\Delta}), \ast([\circ (\varphi)^{\triangleright}]_{\Delta})) =P^{\mathfrak{A}_{\Delta}}((t_1)^{\triangleright}, \ldots, (t_n)^{\triangleright}) =$ \\ $ P^{\mathfrak{A}_{\Delta}}(\termvalue{t_1}^{\widehat{\mathfrak{A}_{\Delta}}},\ldots,\termvalue{t_n}^{\widehat{\mathfrak{A}_{\Delta}}}) $.\\[1mm]
(ii) $v_{\Delta}(\neg\psi) = (\ast([\neg(\psi)^\triangleright]_{\Delta}), \ast([\neg\neg (\psi)^\triangleright ]_{\Delta}), \ast([\circ \neg (\psi)^\triangleright]_{\Delta})) \in \tilde{\neg} v_{\Delta}(\psi)$. On the other hand, $v_{\Delta}(\circ\psi) = (\ast([\circ(\psi)^\triangleright]_{\Delta}), \ast([\neg{\circ} (\psi)^\triangleright ]_{\Delta}), \ast([{\circ} {\circ} (\psi)^\triangleright]_{\Delta})) \in \tilde{\circ} v_{\Delta}(\psi)$.\\[1mm]
(iii) Since $\ast([\delta \# \psi]_{\Delta})=\ast([\delta]_{\Delta})\,\#\, {\ast}([\psi]_{\Delta})$, then $v_{\Delta}(\delta \# \psi) \in v(\delta) \tilde{\#} v_{\Delta}(\psi)$, for every $\#\in \{\wedge,\vee, \to\}$.\\[1mm]
(iv) By Lemma~\ref{quantOK} (and recalling that $U=CTer(\Theta)$),
$[\forall x \psi]_\Delta = \bigwedge_{\mA_\Delta} \{ [\psi[x/t]]_\Delta \ :  \ t \in U\}$ and so $\ast([\forall x \psi]_\Delta) = \bigwedge_{C\mA_{\Delta}} \{ \ast([\psi[x/t]]_\Delta) \ :  \ t \in U\}$. Then, $(v_{\Delta}(\forall x\psi))_1 = {\ast}([(\forall x\psi)^\triangleright]_{\Delta}) = \bigwedge_{t \in U} \ast([(\psi[x/\bar t])^\triangleright]_\Delta) = \bigwedge_{t \in U} (v_\Delta(\psi[x/\bar t]))_1$. \\[1mm] 
(v) The case $\exists x \psi$  is treated analogously.\\[1mm]
(vi) Let $t$ be free for $z$ in $\varphi$, $\mu$ an assignment and  $b = \termvalue{t}^{\widehat{\mathfrak{A}_\Delta}}_{\mu}$. By induction on the complexity of $\varphi$ it can be proved that $(\widehat{\mu}(\varphi[z/t]))^{\triangleright}= (\widehat{\mu}(\varphi[z/\bar{b}]))^{\triangleright}$. Hence,  $v_\Delta(\widehat{\mu}(\varphi[z/t])) = v_\Delta(\widehat{\mu}(\varphi[z/\bar{b}]))$, by definition of $v_\Delta$. From this, it is obvious that  $v_{\Delta}$ satisfies  conditions (vi.1)-(vi.3).\\[1mm]
(vii) If $\varphi$ and $\varphi'$  are variant, so are  $(\varphi)^\triangleright$ and $(\varphi')^\triangleright$; $(\neg\varphi)^\triangleright$ and $(\neg\varphi')^\triangleright$; and $(\circ\varphi)^\triangleright$ and $(\cons\varphi')^\triangleright$. From this,  $v_\Delta(\varphi) = v_\Delta(\varphi')$, by axiom {\bf (Ax14)}.
\end{proof}

\begin{theorem} {\em (Completeness of \qmbc\  restricted to sentences w.r.t. first-order swap structures)} \label{comp-sent-Qmbc-swap}
Let $\Gamma \cup \{\varphi\} \subseteq$ \sent. If $\Gamma \models_{\qmbc} \varphi$ then $\Gamma \vdash_{\qmbc} \varphi$. 
\end{theorem}
\begin{proof}
Let  $\Gamma \cup \{\varphi\} \subseteq$ \sent\ such that $\Gamma \nvdash_{\qmbc} \varphi$. Then, by recalling Definition~\ref{Henkin} and by Theorem~7.5.3 in~\cite{CC16},\footnote{The proof of Theorem~7.5.3 in~\cite{CC16} also holds for the notion of $C$-Henkin theory adopted here in Definition~\ref{Henkin} (which is stronger than the one proposed in~\cite[Definition~7.5.1]{CC16}), as it can be easily verified.} there exists a $C$-Henkin theory $\Delta^{H}$ over $\Theta_{C}$ in \qmbc\ for a nonempty set $C$ of new individual constants such that $\Gamma \subseteq \Delta^{H}$ and, for every $\alpha \in $ \sent:  $\Gamma \vdash_{\qmbc} \alpha$ iff $\Delta^{H} \vdash^{C}_{\qmbc} \alpha$. Hence, $\Delta^{H} \nvdash^{C}_{\qmbc} \varphi$. Now, by Proposition~\ref{saturated}, there exists a set of sentences $\overline{\Delta^{H}}$ in $\Theta_{C}$ extending $\Delta^{H}$ which is maximally non-trivial with respect to $\varphi$  in \qmbc\ (defined over $Sen(\Theta_C)$), such that $\overline{\Delta^{H}}$ is a $C$-Henkin theory over $\Theta_{C}$ in \qmbc. Let $\mathcal{M}(\mathcal{B}_{\overline{\Delta^{H}}})$, $\mathfrak{A}_{\overline{\Delta^{H}}} $ and $v_{\overline{\Delta^{H}}}$ be as in Definitions~\ref{nma},~\ref{str} and~\ref{valca},  respectively. Then,  $v_{\overline{\Delta^{H}}}(\alpha) \in D_{\overline{\Delta^{H}}} \mbox{ iff } \overline{\Delta^{H}} \vdash^{C}_{\qmbc} \alpha$, for every $\alpha$ in $Sen(\Theta_{C})$. From this, $v_{ \overline{\Delta^{H}}}[\Gamma] \subseteq D_{\overline{\Delta^{H}}}$ and $v_{\overline{\Delta^{H}}}(\varphi) \notin D_{\overline{\Delta^{H}}}$. Finally, let $\mathfrak{A}$ and $v$ be the respective restrictions of $\mathfrak{A}_{\overline{\Delta^{H}}} $ and $v_{\overline{\Delta^{H}}}$ to $\Theta$. Then, $\mathfrak{A}$ is a structure over $\mathcal{M}(\mathcal{B}_{\overline{\Delta^{H}}})$, and $v$ is a valuation for \qmbc\ over $\mathfrak{A}$ and $\mathcal{M}(\mathcal{B}_{\overline{\Delta^{H}}})$ such that $v[\Gamma] \subseteq D_{\overline{\Delta^{H}}}$ but $v(\varphi) \notin D_{\overline{\Delta^{H}}}$. This shows that $\Gamma \not\models_{\qmbc} \varphi$.
\end{proof}

\noindent
For any formula $\psi$ in $For(\Theta)$ let $(\forall)\psi$ be the {\em universal closure} of $\psi$, defined as follows: if $\psi$ is a sentence then  $(\forall)\psi \defin \psi$; and if $\psi$ has exactly the variables $x_1,\ldots,x_n$ occurring free then  $(\forall)\psi \defin (\forall x_1)\cdots(\forall x_n)\psi$. Note that $(\forall)\psi \in Sen(\Theta)$. If $\Gamma$ is a set of formulas in $For(\Theta)$ then  $(\forall)\Gamma \defin \{(\forall)\psi  \ : \  \psi \in \Gamma\}$. It is easy to show that,   for every $\Gamma \cup \{\varphi\} \subseteq For(\Theta)$:  (i)~$\Gamma\vdash_{\qmbc} \varphi$ \ iff \   $(\forall)\Gamma\vdash_{\qmbc} (\forall)\varphi$; and (ii)~$\Gamma \models_{\qmbc} \varphi$ \ iff \   $(\forall)\Gamma \models_{\qmbc} (\forall)\varphi$. From this, a general completeness result can be obtained:

\begin{corollary} {\em (Completeness of \qmbc\ w.r.t. first-order swap structures)} \label{comp-QmbC} Let $\Gamma \cup \{\varphi\} \subseteq For(\Theta)$.  If $\Gamma \models_{\qmbc} \varphi$ then $\Gamma\vdash_{\qmbc} \varphi$.
\end{corollary}

\section{Completeness of \qmbc\ w.r.t. structures over $\matM_5$} \label{compM5}

Recall from Section~\ref{M5} the 5-valued Nmatrix $\matM_5$ introduced by Avron in~\cite{avr:05}. From the adequacy of \qmbc\ w.r.t. first-order swap structures, and given that \mbc\ can be characterized just with  $\matM_5 = \mathcal{M}\big(\mathcal{B}_{\mA_2}\big)$, it is a natural question to determine if it is possible to extend the proof of~\cite[Theorem~6.4.9 and Corollary~6.4.10]{CC16} (see Theorem~\ref{val-bival-mbC} above) to \qmbc. Namely, taking into account that \qmbc\ can be characterized by standard Tarskian structures expanded with bivaluations which naturally extend the ones for \mbc\ (see Theorem~\ref{comple2Qmbc} below), it seems plausible to extend the technique of  Theorem~\ref{val-bival-mbC} to \qmbc. In Theorem~\ref{adeq-qmbcM5} below it will be shown that this is really the case, hence \qmbc\ can be characterized by first-order structures  over $\matM_5$. Such structures, which were introduced by Avron and Zamansky in~\cite{avr:zam:07} (see Remark~\ref{defquantM5} below), can be considered as  being `classical', as discussed  at the end of Section~\ref{M5}.

Consider a standard Tarskian first-order structure $\mathsf{A} = \langle U,I_{\mathsf{A}} \rangle$ over a first-order signature $\Theta$ (see, for instance, \cite{mendelson}). Observe that $\mathsf{A}$ is defined as in Definition~\ref{stru}, but now any predicate symbol $P$ of arity $n$ is interpreted as a subset  $I_\mathsf{A}(P)$ of $U^n$. The notions of  diagram language $For(\Theta_U)$, extended structure $\widehat{\mathsf{A}}$ and  \sena\ are defined as in Definitions~\ref{diaglan} and~\ref{extA}, and Notation~\ref{SentA} above. 

In~\cite{CC16} the notion of bivaluations for \mbc\  was extended to {\em bivaluations for \qmbc} as follows:

\begin{definition} (Bivaluations for \qmbc, \cite[Definition~7.3.5]{CC16}) \label{bivalqmbc}
Let $\mathsf{A}=\langle U, I_{\mathsf{A}} \rangle$ be a standard Tarskian first-order structure  over $\Theta$, and let $\widehat{\mathsf{A}}=\langle U, I_{\widehat{\mathsf{A}}} \rangle$ be the expansion of $\mathsf{A}$ to $\Theta_U$ by setting $I_{\widehat{\mathsf{A}}}(\bar{a})=a$ for every $a \in U$. A {\em bivaluation\footnote{It was called \qmbc-valuation in~\cite[Definition~7.3.5]{CC16}.} for \qmbc\ over $\mathsf{A}$} is a function $\rho:\sena \to \{0,1\}$ satisfying the clauses of  Definition~\ref{bivalold} above plus the following:\\[2mm]
{\bf (\valpred)} \ $\rho(P(t_1,\ldots,t_n)) = 1$
 \  iff \ 
$\langle \termvalue{t_1}^{\widehat{\mathsf{A}}},\ldots,\termvalue{t_n}^{\widehat{\mathsf{A}}}
\rangle \in I_{\mathsf{A}}(P)$, for
$P(t_1,\ldots,t_n)$ atomic \\[2mm]
{\bf (\valvar)} \ $\rho(\varphi) = \rho(\psi) ~\mbox{whenever $\varphi$
is a variant of $\psi$}$ \\[2mm]
{\bf (\valuni)} \ $\rho(\forall x \varphi)=1$ \  iff \ 
$\rho(\varphi[x/\bar{a}])=1 ~\mbox{for every $a \in U$}$ \\[2mm]
{\bf (\valex)} \ $\rho(\exists x \varphi) = 1$ \  iff \ 
$\rho(\varphi[x/\bar{a}])=1 ~\mbox{for some $a \in U$}$ \\[2mm]
{\bf (\valsubs)} \ if $\mu$ is an assignment, $t$ is a term free for $z$ in $\varphi$ and $b = \termvalue{t}^{\widehat{\mathsf{A}}}_{\mu}$, then: $\rho(\widehat{\mu}(\varphi[z/t])) = \rho(\widehat{\mu}(\varphi[z/\bar{b}])$ implies $\rho(\widehat{\mu}(\#\varphi[z/t])) = \rho(\widehat{\mu}(\#\varphi[z/\bar{b}]))$, for $\# \in \{\neg,\cons\}$.
\end{definition}

\begin{definition}  (\cite[Definitions~7.3.6 and~7.3.12]{CC16}) \label{sem2Qmbc}
An {\em interpretation} for \qmbc\ over a signature $\Theta$ is a pair $\langle \mathsf{A}, \rho\rangle$ such that $\mathsf{A}$ is a Tarskian first-order structure over $\Theta$ and $\rho$ is a bivaluation for \qmbc\ over  $\mathsf{A}$.  The {\em consequence relation $\models_{\qmbc}^2$ of \qmbc\ w.r.t. interpretations} is given by: $\Gamma\models_{\qmbc}^2 \varphi$ if, for every interpretation $\langle \mathsf{A}, \rho\rangle$: $\rho(\widehat{\mu}(\gamma))=1$ for every $\gamma \in \Gamma$ and every assignment $\mu$ implies that $\rho(\widehat{\mu}(\varphi))=1$ for every assignment $\mu$.
\end{definition}

\begin{theorem} {\em (Adequacy of \qmbc\ w.r.t. interpretations, \cite[Theorems~7.4.1. and~7.5.6]{CC16})} \label{comple2Qmbc} If $\Gamma \cup \{\varphi\} \subseteq For(\Theta)$ then:  $\Gamma\vdash_{\qmbc} \varphi$ \ iff \ $\Gamma\models_{\qmbc}^2 \varphi$.\footnote{Rigourously speaking, in~\cite[Theorem~7.5.6]{CC16} it was obtained completeness of \qmbc\ w.r.t. interpretations, but only for sentences. However, completeness of \qmbc\ (in the full language) w.r.t. interpretations follows easily, as we have done here in Corollary~\ref{comp-QmbC}.}
\end{theorem}

\noindent Now, Theorem~\ref{val-bival-mbC} will be extended to \qmbc\ (see Theorem~\ref{val-bival-qmbC} below). Previous to this, it is worth observing the following:

\begin{remark} \label{defquantM5} Consider once again the characteristic 5-valued Nmatrix  $\matM_5 = \mathcal{M}\big(\mathcal{B}_{\mA_2}\big)$ for \mbc, and let $\mathfrak{A}$ be a structure over $\Theta$ and $\matM_5$. If $v$ is a valuation for \qmbc\ over $\mathfrak{A}$ and $\matM_5$, it is easy to see that clauses~(iv) and~(v) of Definition~\ref{val} are equivalent to the following:  for every $Q\in \{\forall, \exists\}$, $v(Qx\varphi) \in \tilde{Q}\big(\{v(\varphi[x/\bar{a}]) \ : \ a \in U\}\big)$ where $\tilde{Q}: (\mathcal{P}(B_{\mA_2})-\{\emptyset\}) \to (\mathcal{P}(B_{\mA_2})-\{\emptyset\})$ is
$$\tilde{\forall}(X) =  \begin{cases} \textrm{D}, \mbox{ if }  X\subseteq \textrm{D}\\[1mm] 
\textrm{ND}, \mbox{ otherwise}   \end{cases} \ \mbox{ and } \ \ \  \tilde{\exists}(X) =  \begin{cases} \textrm{D}, \mbox{ if }  X\cap \textrm{D} \neq \emptyset \\[1mm] 
\textrm{ND}, \mbox{ otherwise}   \end{cases}   $$
and
$B_{\mA_2} = |\matM_5| = \big\{T, \, t, \, t_0, \, F, \, f_0\big\}$,  $\textrm{D}=\{T, \, t, \, t_0\}$  and $\textrm{ND}=\{F, \, f_0\}$  (recall Section~\ref{M5}). It is not hard to prove that the notions of structures over $\Theta$ and $\matM_5$, and valuations over them, coincide with the corresponding notions introduced in~\cite{avr:zam:07}. Thus, the present framework generalizes, from $\mA_2$ to arbitrary complete Boolean algebras, the semantical framework proposed in~\cite{avr:zam:07}.
\end{remark}

\begin{theorem}  \label{val-bival-qmbC}
Let  $\mathcal{I}=\langle \mathsf{A}, \rho\rangle$ be an interpretation for \qmbc\ over a signature $\Theta$. Then, it induces a first-order structure $\mathfrak{A}_\mathcal{I}$ over $\matM_5$ and $\Theta$, and a \qmbc-valuation $v^\rho$ over  $\mathfrak{A}_\mathcal{I}$ and $\matM_5$  given by $v^\rho(\alpha) \defin (\rho(\alpha),\rho(\neg\alpha),\rho(\cons\alpha))$ such that: $\rho(\alpha)=1$ iff $v^\rho(\alpha) \in {\rm D}$, for every sentence $\alpha \in \sena$.
\end{theorem}
\begin{proof} Given $\mathcal{I}=\langle \mathsf{A}, \rho\rangle$ consider the first-order structure $\mathfrak{A}_\mathcal{I}$ over $\matM_5$ and $\Theta$ obtained from $\mathsf{A}$ by taking the same domain $U$; $I_{\mathfrak{A}_\mathcal{I}}$ coincides with $I_\mathsf{A}$ for every individual constant and function symbol; and  $I_{\mathfrak{A}_\mathcal{I}}(P): U^n \to |\matM_5|$ is given by $I_{\mathfrak{A}_\mathcal{I}}(P)(a_1,\ldots,a_n)= v^\rho(P(\bar{a}_1,\ldots,\bar{a}_n))$ for every predicate symbol $P$ of arity $n$,
where $v^\rho:\sena \to |\matM_5|$ is defined by $v^\rho(\alpha) = (\rho(\alpha),\rho(\neg\alpha),\rho(\cons\alpha))$, for every  $\alpha \in \sena$. Clearly $\rho(\alpha)=1$ iff $v^\rho(\alpha) \in {\rm D}$, for every  $\alpha \in \sena$. Thus, it remains to prove that   $v^\rho$ is indeed a \qmbc-valuation over  $\mathfrak{A}_\mathcal{I}$ and $\matM_5$. It is clear that clauses~(i)-(iii) of Definition~\ref{val} are satisfied, since $\termvalue{t}^{\widehat{\mathfrak{A}_\mathcal{I}}}=\termvalue{t}^{\widehat{\mathsf{A}}}$ for every closed term $t$, and by  Theorem~\ref{val-bival-mbC}. With respect to clause~(iv), suppose
that $\{v^\rho(\varphi[x/\bar{a}]) \ : \ a \in U\} \subseteq {\rm D}$. Then $\rho(\varphi[x/\bar{a}])=1$ for every $a \in U$ and so $\rho(\forall x \varphi)=1$, by (\valuni). Hence $v^\rho(\forall x\varphi) \in {\rm D}$. This means that $v^\rho(\forall x\varphi) \in
 \tilde{\forall}\big(\{v^\rho(\varphi[x/\bar{a}]) \ : \ a \in U\}\big)$. If $\{v^\rho(\varphi[x/\bar{a}]) \ : \ a \in U\} \not\subseteq {\rm D}$ then, by a similar reasoning, it is shown that, once again, $v^\rho(\forall x\varphi) \in
 \tilde{\forall}\big(\{v^\rho(\varphi[x/\bar{a}]) \ : \ a \in U\}\big)$. Analogously it can be proven that $v^\rho$ satisfies clause~(v).  Clause~(vi) is satisfied by $v^\rho$, since $\termvalue{t}^{\widehat{\mathfrak{A}_\mathcal{I}}}_{\mu}=\termvalue{t}^{\widehat{\mathsf{A}}}_{\mu}$ for every term $t$, and by the fact that $\rho$ satisfies the Substitution Lemma: $\rho(\widehat{\mu}(\varphi[z/t])) = \rho(\widehat{\mu}(\varphi[z/\bar{b}]))$ for $b = \termvalue{t}^{\widehat{\mathsf{A}}}_{\mu}$. Clause~(vii) is also satisfied, since $\rho$ satisfies (\valvar). This concludes the proof.
\end{proof}

\begin{theorem} {\em (Adequacy of \qmbc\ w.r.t. first-order structures over $\matM_5$)} \label{adeq-qmbcM5}
For every set $\Gamma \cup \{\varphi\} \subseteq For(\Theta)$:  $\Gamma\vdash_{\qmbc} \varphi$ \ iff \ $\Gamma\models_{(\mathfrak{A}, \matM_5)} \varphi$ for every  structure $\mathfrak{A}$ over $\Theta$ and $\matM_5$.
\end{theorem}
\begin{proof} \ \\
`Only if' part (Soundness):  It is a consequence of Theorem~\ref{sound-Qmbc-swap}.\\[1mm]
`If' part (Completeness): Suppose that $\Gamma\models_{(\mathfrak{A}, \matM_5)} \varphi$ for every  structure $\mathfrak{A}$ over $\Theta$ and $\matM_5$. Let $\mathcal{I}=\langle \mathsf{A}, \rho\rangle$ be an interpretation for \qmbc\ over $\Theta$ such that $\rho(\widehat{\mu}(\gamma))=1$ for every $\gamma \in \Gamma$ and every assignment $\mu$. Let $\mathfrak{A}_\mathcal{I}$ and $v^\rho$ be as in Theorem~\ref{val-bival-qmbC}. Then $v^\rho_{\mu}(\gamma)\in {\rm D}$,  for every formula $\gamma \in \Gamma$ and every assignment $\mu$. By hypothesis,  $v^\rho_{\mu}(\varphi)\in {\rm D}$,  for every assignment $\mu$. This implies that  $\rho(\widehat{\mu}(\varphi))=1$, for every assignment $\mu$. That is, $\Gamma\models_{\qmbc}^2 \varphi$. Therefore $\Gamma\vdash_{\qmbc} \varphi$, by Theorem~\ref{comple2Qmbc}.
\end{proof}

\noindent
Taking into account Remark~\ref{defquantM5}, the last result  is  a restatement of the adequacy for \qmbc\ w.r.t. first-order structures over $\matM_5$ obtained in~\cite[Theorem~24]{avr:zam:07}.

\section{Adding standard equality to \qmbc} \label{secEq}

In this section a binary predicate $ \approx$ for dealing with equality will be considered. As expected, this predicate will be always interpreted as the standard identity. This means that the predicate $\approx$ will be seen, from a semantical point of view, as
a logical symbol. The  resulting logic will be called $\qmbc^\approx$. The definition of \qmbceq\ will follows closely~\cite[Section~7.7]{CC16}. 

\begin{definition}
Let $\Theta$ be a first-order signature. The induced signature with equality $\Theta_\approx$ is obtained from $\Theta$ by adding a new binary predicate symbol $ \approx$.
\end{definition}

\noindent The expression $(t_1 \approx t_2)$ will stands for the atomic formula $\approx(t_1,t_2)$.
If $\varphi$ is a formula and $y$ is a variable free for
the variable $x$ in $\varphi$,  $\varphi[x\wr y]$ denotes any
formula obtained from $\varphi$ by replacing some, but not
necessarily all (maybe none), free occurrences of $x$ by $y$.

\begin{definition} (\cite[Definition 7.7.1]{CC16}) \label{DefQmbCEq} Let $\Theta_\approx$ be a first-order signature with equality. The logic \qmbceq\ (over $\Theta_\approx$) is the
extension  of \qmbc\ over $For(\Theta_\approx)$ obtained by adding to
\qmbc, besides all the new instances of axioms and inference
rules involving the equality predicate $\approx$, the following
axiom schemas:\\[2mm]
$\begin{array}{ll}
{\bf (AxEq1)} & \forall x(x \approx x)\\[2mm]
{\bf (AxEq2)} & (x \approx y) \rightarrow (\varphi \rightarrow \varphi[x \wr y]) \mbox{, if $y$ is a variable free
for $x$ in $\varphi$}
\end{array} $
\end{definition}

\noindent Notice that the axioms for equality are the same considered for classical logic (see, for instance, \cite{mendelson}).  
Given that \qmbceq\ is an axiomatic extension of \qmbc, it satisfies the deduction meta-theorem DMT (recall Theorem~\ref{teoded:teo}).
Let $\vdash_{\qmbceq}$ be the consequence relation of the Hilbert calculus \qmbceq. The semantics of first-order swap structures for \qmbc\ can be easily  extended to the equality predicate.

\begin{definition} \label{strueq} Let ${\cal M}({\cal B})=(\mathcal{B},D)$ be a non-deterministic matrix defined by a swap structure $\mathcal{B}$ for {\bf mbC}, and let $\Theta_\approx$ be a first-order signature with equality. A  (first-order) {\em structure with  standard equality} over ${\cal M}({\cal B})$ and $\Theta_\approx$ is a structure over ${\cal M}({\cal B})$ and $\Theta_\approx$ such that
$I_\mathfrak{A}(\approx)(a,b) \in D \ \mbox{ iff } \ a=b$.
\end{definition}

\noindent In what follows, $(a \approx^\mathfrak{A}b)$ will stands for $I_\mathfrak{A}(\approx)(a,b)$, for every structure $\mathfrak{A}$ and any $a,b \in U$. Given a structure $\mathfrak{A}$, the signature obtained from $\Theta$ by adding a  new individual constant for each element of $U$  (recall Definition~\ref{diaglan}) will be denoted by $\Theta^ \approx_U$. The set of formulas and sentences over  $\Theta^ \approx_U$ will be denoted by $For(\Theta^ \approx_U)$ and $Sen(\Theta^ \approx_U)$, respectively.

If $\mathfrak{A}$ is a structure with standard equality over  ${\cal M}({\cal B})$ and $v$ is a \qmbc-valuation over $\mathfrak{A} $ and  ${\cal M}({\cal B})$ then, by Definition~\ref{val}~(i) it follows that, for every closed terms $t_1$ and $t_2$, $v(t_1 \approx t_2) \in D \ \mbox{ iff } \ \termvalue{t_1}^\mathfrak{A} = \termvalue{t_2}^\mathfrak{A}$.
This guarantees the validity of  axiom~({\bf AxEq1}). However, in order to validate axiom~({\bf AxEq2}), the valuations must be additionally restricted:

\begin{definition} (\qmbceq-valuations)~\label{valeq}
Let  $\mathfrak{A} $ be a structure with standard equality over $\Theta_\approx$ and ${\cal M}({\cal B})$. A {\em valuation} for \qmbceq\ (or a \qmbceq-{\em valuation}) over $\mathfrak{A} $ and  ${\cal M}({\cal B})$ is a \qmbc-valuation $v:Sen(\Theta^ \approx_U) \to |{\cal B}|$ which satisfies, in addition, the following clause, for every $\mu$:\\[2mm]
($viii$) $v_\mu((x \approx y) \rightarrow (\varphi \rightarrow \varphi[x \wr y])) \in D$, if $y$ is a variable free for $x$ in $\varphi$.
\end{definition}

\noindent The notion of  $\varphi$ being a {\em $\approx$-semantical consequence of $\Gamma$ over $(\mathfrak{A}, {\cal M}({\cal B}))$}, denoted by $\Gamma\models^\approx_{(\mathfrak{A}, {\cal M}({\cal B}))}\varphi$, is as stated in Definition~\ref{semcons}, but now restricted to structures with standard equality and \qmbceq-valuations over them. Thus, $\varphi$ is  a {\em semantical consequence of $\Gamma$ in \qmbceq\ w.r.t. first-order swap structures}, denoted by $\Gamma\models_{\qmbceq}\varphi$, if $\Gamma\models^\approx_{(\mathfrak{A}, {\cal M}({\cal B}))}\varphi$ for every of such pairs $(\mathfrak{A}, {\cal M}({\cal B}))$.

Observe that the Substitution Lemma still holds for \qmbceq, since  it holds for any structure and any \qmbc-valuation. From this, and from Definition~\ref{valeq}, the following result can be easily derived by adapting the proof of Theorem~\ref{sound-Qmbc-swap}: 

\begin{theorem} {\em (Soundness of \qmbceq\ w.r.t. first-order swap structures  with standard equality)} \label{sound-Qmbceq-swap}
For every set $\Gamma \cup\{ \varphi\}  \subseteq For(\Theta_\approx)$:  if $\Gamma \vdash_{\qmbceq} \varphi$, then  $\Gamma \models_{\qmbceq} \varphi$.
\end{theorem}

\noindent In order to prove completeness of \qmbceq\ w.r.t. swap structures semantics, the proof given in Section~\ref{complete} will be adapted, in accordance with the argument given in~\cite[Section~7.7]{CC16}.

We begin by observing that the notion of $C$-Henkin theory in \qmbceq\ can be defined by adapting Definition~\ref{Henkin} in an obvious way. The signature obtained from $\Theta_\approx$ by adding a set $C$ of new individual constants will be denoted by $\Theta_C^\approx$, and the consequence relation in \qmbceq\ over that signature will be denoted by $\vdash^{C}_{\qmbceq}$. Clearly, Proposition~\ref{saturated} also holds for \qmbceq.  This result, combined with ~\cite[Theorem~7.5.3]{CC16} (which can also be easily adapted to \qmbceq) produces the following: 

\begin{proposition} \label{saturatedeq}
Let  $\Gamma \cup \{\varphi\} \subseteq Sen(\Theta_\approx)$  such that $\Gamma \nvdash_{\qmbceq} \varphi$. Then, there exists a set of sentences $\Delta \subseteq Sen(\Theta_C^\approx)$, for some set $C$ of new individual constants, such that $\Gamma \subseteq \Delta$, it is a $C$-Henkin theory in \qmbceq, and it is maximally non-trivial with respect to $\varphi$ in \qmbceq\ (by restricting $\vdash^C_{\qmbceq}$ to sentences in $Sen(\Theta_C^\approx)$).
\end{proposition}

\begin{definition} Let $\Delta \subseteq Sen(\Theta_\approx)$ be non-trivial in \qmbceq, that is: there is some sentence $\varphi$ in $Sen(\Theta_\approx)$ such that $\Delta \nvdash_{\qmbceq} \varphi$. Let ${\equiv_{\Delta}^\approx} \subseteq Sen(\Theta_\approx)^{2}$ be the relation in $Sen(\Theta_\approx)$ defined as follows:  $\alpha \equiv_{\Delta}^\approx \beta$ iff $\Delta \vdash_{\qmbceq} \alpha \leftrightarrow\beta$. 
\end{definition}

\noindent
Then $\equiv_{\Delta}^\approx$ is an equivalence relation which induces a Boolean algebra  $\mathcal{A}_{\Delta}^\approx$ whose domain is the quotient set  $A_{\Delta}^\approx \defin Sen(\Theta_\approx)/_{\equiv_{\Delta}^\approx}$ such that the operations are defined  as follows (here $[\alpha]_\Delta^\approx$ denotes the equivalence class of $\alpha$ w.r.t. $\equiv_{\Delta}^\approx$): $[\alpha]_\Delta^\approx \,\bar\#\, [\beta]_\Delta^\approx \defin [\alpha \# \beta]_\Delta^\approx$ for any $\# \in \{\wedge, \vee, \to\}$; $0_{\Delta}^\approx \defin [\varphi \wedge \neg \varphi \wedge \circ \varphi]_{\Delta}^\approx$, and  $1_{\Delta}^\approx \defin [\varphi \vee \neg \varphi]_{\Delta}^\approx$ (for any sentence $\varphi$).

\begin{definition}\label{nmaeq}
 Let $\mathcal{A}_{\Delta}^\approx$ be a Boolean algebra defined as above, and let $C\mA_\Delta^\approx$ be its completion with monomorphism $\ast$ (recall Section~\ref{complete}). The full swap structure for \mbc\ over $C\mathcal{A}_{\Delta}^\approx$ will be denoted by  $\mathcal{B}_{\Delta}^\approx$, and the associated Nmatrix will be denoted by $\mathcal{M}(\mathcal{B}_{\Delta}^\approx) \defin (\mathcal{B}_{\Delta}^\approx, D_{\Delta}^\approx)$. 
\end{definition}

\begin{remark}\label{rem1eq}
Note that $(\ast([\alpha]_{\Delta}^\approx), \ast([\beta]_{\Delta}^\approx), \ast([\gamma]_{\Delta}^\approx)) \in D_{\Delta}^\approx \ \mbox{ iff } \ \Delta \vdash_{\qmbceq} \alpha$.
\end{remark}

\begin{definition}  (Canonical Structure in \qmbceq) \label{streq}
Let $\Delta \subseteq Sen(\Theta_C^\approx)$ be a non-trivial and $C$-Henkin theory in \qmbceq.  Define in the set $C$ of constants the following relation:
$c \simeq d$ iff $\Delta \vdash_{\qmbceq}^{C} (c \approx d)$. By the axioms of equality it follows that $\simeq$ is an equivalence relation. For any $c \in C$ let $\widetilde{c} =
\{d \in C \ : \ c \simeq d\}$ be the equivalence class of $c$ w.r.t. $\simeq$, and let $U =
\{\widetilde{c} \ : \ c \in C\}$ be the corresponding quotient set. Let $\mathcal{M}(\mathcal{B}_{\Delta}^\approx)$ be as in Definition~\ref{nmaeq}. The {\em canonical structure induced by $\Delta$ in \qmbceq} is the  structure  $\mathfrak{A}_\Delta^\approx = \langle U, I_{{\mathfrak{A}_\Delta^\approx}} \rangle$
over  $\Theta_C^\approx$ and $\mathcal{M}(\mathcal{B}_{\Delta}^\approx)$ defined as follows:
\begin{itemize}
\item[-] if $c$ is an individual constant in $\Theta_C^\approx$ then $I_{{\mathfrak{A}_\Delta^\approx}}(c)=
  \widetilde{d}$, where $d \in C$ is such that $\Delta
  \vdash_{\qmbceq}^{C} (c\approx d)$;
  \item[-] if $f$ is a function symbol,  $I_{{\mathfrak{A}_\Delta^\approx}}(f):U^n \to U$ is given by $I_{{\mathfrak{A}_\Delta^\approx}}(f)(\widetilde{c}_1,\ldots,\widetilde{c}_n)= \widetilde{c}$, where $c \in C$ is such that $\Delta
  \vdash_{\qmbceq}^{C} (f(c_1,\ldots,c_n)\approx c)$;
  \item[-] if $P$ is a predicate symbol, then $I_{{\mathfrak{A}_\Delta^\approx}}(P)$ is given by
 \item[]$I_{{\mathfrak{A}_\Delta^\approx}}(P)({\widetilde{c}_1},\ldots,{\widetilde{c}_n}) = (\ast([P(c_1,\ldots,c_n)]_{\Delta}^\approx), \ast([\neg P(c_1,\ldots,c_n)]_{\Delta}^\approx), \ast([\cons P(c_1,\ldots,c_n)]_{\Delta}^\approx))$.
\end{itemize}
\end{definition}

\noindent 
The proof that $I_{{\mathfrak{A}_\Delta^\approx}}$ is well-defined for individual constants and function symbols is similar to that for classical logic (see~\cite{chan:keis}). Let $P$ be predicate symbol of arity $n$ and let  $({\widetilde{c}_1},\ldots,{\widetilde{c}_n}) \in U^n$. Let $d_1,\ldots,d_n \in C$ such that $c_i \simeq d_i$ for every $1 \leq i \leq n$. Then $\Delta \vdash_{\qmbceq}^{C} (c_i \approx d_i)$ for $1 \leq i \leq n$. It is immediate that $\Delta\vdash_{\qmbceq}^{C} \left(\bigwedge_{i=1}^n(c_i \approx d_i)\right) \to (P(c_1,\ldots,c_n) \leftrightarrow P(d_1,\ldots,d_n))$, hence $\Delta\vdash_{\qmbceq}^{C} (P(c_1,\ldots,c_n) \leftrightarrow P(d_1,\ldots,d_n))$. Analogously it can be proven that $\Delta\vdash_{\qmbceq}^{C} (\neg P(c_1,\ldots,c_n) \leftrightarrow \neg P(d_1,\ldots,d_n))$ and $\Delta\vdash_{\qmbceq}^{C} (\cons P(c_1,\ldots,c_n) \leftrightarrow \cons P(d_1,\ldots,d_n))$. This shows that $I_{{\mathfrak{A}_\Delta^\approx}}(P)$ is well-defined. Moreover, by similar considerations to the ones given after Definition~\ref{str}, it follows that $I_{{\mathfrak{A}_\Delta^\approx}}(P)({\widetilde{c}_1},\ldots,{\widetilde{c}_n})$ belongs to $|\mathcal{B}_{\Delta}^\approx|$ for every $({\widetilde{c}_1},\ldots,{\widetilde{c}_n}) \in U^n$.

\begin{proposition} \label{is-stru}
Let $\Delta \subseteq Sen(\Theta_C^\approx)$ be a non-trivial and $C$-Henkin theory in \qmbceq\ and let $\mathcal{M}(\mathcal{B}_{\Delta}^\approx)$ be as in Definition~\ref{nmaeq}. Then the canonical structure  $\mathfrak{A}_\Delta^\approx$ induced by $\Delta$ in \qmbceq is a  structure  with standard equality
over  $\Theta_C^\approx$ and $\mathcal{M}(\mathcal{B}_{\Delta}^\approx)$. 
\end{proposition}
\begin{proof}
As it was shown above, the mapping  $I_{{\mathfrak{A}_\Delta^\approx}}$ is well-defined, and $I_{{\mathfrak{A}_\Delta^\approx}}(P)$ is a function from $U^n$ to $|\mathcal{B}_{\Delta}^\approx|$. Let $\widetilde{c}_1,\widetilde{c}_2 \in U$. Then $(\widetilde{c}_1 \approx^{\mathfrak{A}_\Delta^\approx}{\widetilde{c}_2}) \in D_{\Delta}^\approx$ iff $\Delta\vdash_{\qmbceq}^{C} (c_1\approx c_2)$, by Definition~\ref{streq} and by Remark~\ref{rem1eq}, iff $c_1 \simeq c_2$ iff  $\widetilde{c}_1 = {\widetilde{c}_2}$. This shows that $\mathfrak{A}_\Delta^\approx$ is indeed a  structure  with standard equality.
\end{proof}

\begin{definition} \label{translaeq}
Let $(\cdot)^\triangleleft:(Ter((\Theta_C^\approx)_U) \cup For((\Theta_C^\approx)_U)) \to (Ter(\Theta_C^\approx) \cup For(\Theta_C^\approx))$ 
be the mapping recursively defined  as in Definition~\ref{transla}, but with the following difference: $\left(\,\overline{\widetilde{c}}\,\right)^\triangleleft = d$ for some $d \in \widetilde{c}$ previously chosen, for every $\widetilde{c} \in U$.
\end{definition}

\noindent
Observe that, if $s \in Ter((\Theta_C^\approx)_U) \cup For((\Theta_C^\approx)_U)$, then $(s)^\triangleleft$ is the expression in $Ter(\Theta_C^\approx) \cup For(\Theta_C^\approx)$ obtained from $s$ by substituting every occurrence of a constant $\overline{\widetilde{c}}$ by a constant  $d \in \widetilde{c}$. Suppose that $(\cdot)^{\triangleleft'}:(Ter((\Theta_C^\approx)_U) \cup For((\Theta_C^\approx)_U)) \to (Ter(\Theta_C^\approx) \cup For(\Theta_C^\approx))$ is defined as $(\cdot)^\triangleleft$, but now $\left(\,\overline{\widetilde{c}}\,\right)^{\triangleleft'} = d'$  for another choice of $d' \in \widetilde{c}$ (possibly different to $d$), for every $\widetilde{c} \in U$. Then, it is easy to prove that $\Delta\vdash_{\qmbceq}^{C} (\psi)^\triangleleft \leftrightarrow (\psi)^{\triangleleft'}$ for every sentence $\psi$. This shows that the choice of each $d  \in \widetilde{c}$ in order to define $\left(\,\overline{\widetilde{c}}\,\right)^\triangleleft$, for every $\widetilde{c} \in U$, is irrelevant.

\begin{proposition} \label{valcaeq}
Let $\Delta \subseteq Sen(\Theta_C^\approx)$ be a set of sentences over the signature $\Theta_C^\approx$ such  that $\Delta$ is a $C$-Henkin theory in \qmbceq\ which is also maximally non-trivial with respect to $\varphi$ in \qmbceq, for some sentence $\varphi$. Then, the  canonical \qmbc-valuation induced by $\Delta$ over $\mathfrak{A}_{\Delta}^\approx$ and $\mathcal{M}(\mathcal{B}_{\Delta}^\approx)$ (see Definition~\ref{valca}) is a \qmbceq-valuation, which will be denoted by $v_{\Delta}^\approx$, such that $v_{\Delta}^\approx(\psi) \in D_{\Delta}^\approx \mbox{ iff } \Delta \vdash_{\qmbceq}^{C} (\psi)^\triangleleft$.
\end{proposition}
\begin{proof} Observe that, by the very definitions,  $v_{\Delta}^\approx(\psi) \in D_{\Delta}^\approx \mbox{ iff } \Delta \vdash_{\qmbceq}^{C} (\psi)^\triangleleft$ (and, as observed above, `$v_{\Delta}^\approx(\psi) \in D_{\Delta}^\approx$' does not depend on the choices made by $(\cdot)^\triangleleft$).
Hence, it suffices to prove that $v_{\Delta}^\approx$ satisfies clause~(viii) of Definition~\ref{valeq}. Thus, let $\alpha=(x \approx y) \rightarrow (\psi \rightarrow \psi[x \wr y])$ (where  $y$ is a variable free
for $x$ in $\varphi$)  be an instance of  axiom~({\bf AxEq2}), and let $\mu$ be an  assignment. Given that $\Delta$ is a closed theory in \qmbceq\ over $\Theta_C^\approx$, it follows that $\Delta \vdash_{\qmbceq}^{C}\ (\widehat{\mu}(\alpha))^\triangleleft$. Then $v_{\Delta}^\approx(\widehat{\mu}(\alpha)) \in  D_{\Delta}^\approx$, showing that $v_{\Delta}^\approx$ satisfies clause~(viii).
\end{proof}

\begin{theorem} {\em (Completeness  of \qmbceq\ restricted to  sentences w.r.t. first-order swap structures with standard equality)} \label{comp-sent-Qmbceq-swap}
Let $\Gamma \cup \{\varphi\} \subseteq Sen(\Theta_\approx)$. If $\Gamma \models_{\qmbceq} \varphi$ then $\Gamma \vdash_{\qmbceq} \varphi$. 
\end{theorem}
\begin{proof}
Let  $\Gamma \cup \{\varphi\} \subseteq  Sen(\Theta_\approx)$ such that $\Gamma \nvdash_{\qmbceq} \varphi$. By Proposition~\ref{saturatedeq}, there exists a set of sentences $\Delta \subseteq Sen(\Theta_C^\approx)$, for some set $C$ of new individual constants, such that $\Gamma \subseteq \Delta$, $\Delta$ is a $C$-Henkin theory in \qmbceq, and it is maximally non-trivial with respect to $\varphi$ in \qmbceq\ (by restricting $\vdash^C_{\qmbceq}$ to sentences in $Sen(\Theta_C^\approx)$).

Now, let $\mathcal{M}(\mathcal{B}_{\Delta}^\approx)$, $\mathfrak{A}_\Delta^\approx$ and $v_{\Delta}^\approx$ be as in Definitions~\ref{nmaeq} and~\ref{streq}, and as in Proposition~\ref{valcaeq},  respectively. Then,  $v_{\Delta}^\approx(\alpha) \in D_{\Delta}^\approx \mbox{ iff } \Delta \vdash_{\qmbceq}^{C} \alpha$, for every $\alpha$ in $Sen(\Theta_{C}^\approx)$ (by observing that $(\alpha)^\triangleleft=\alpha$ if $\alpha \in Sen(\Theta_{C}^\approx)$). From this, $v_{\Delta}^\approx[\Gamma] \subseteq D_{\Delta}^\approx$ and $v_{\Delta}^\approx(\varphi) \notin D_{\Delta}^\approx$. Finally, let $\mathfrak{A}$ and $v$ be the restriction to $\Theta_\approx$ of $\mathfrak{A}_\Delta^\approx $ and $v_\Delta^\approx$, respectively. Then, $\mathfrak{A}$ is a structure with standard equality over $\mathcal{M}(\mathcal{B}_\Delta^\approx)$, and $v$ is a valuation for \qmbceq\ over $\mathfrak{A}$ and $\mathcal{M}(\mathcal{B}_\Delta^\approx)$ such that $v[\Gamma] \subseteq D_\Delta^\approx$ but $v(\varphi) \notin D_\Delta^\approx$. This shows that $\Gamma \not\models_{\qmbceq} \varphi$.
\end{proof}

\noindent
By using universal closure, as it was done in Corollary~\ref{comp-QmbC}, a general completeness result can be obtained for \qmbceq:

\begin{corollary} {\em (Completeness of \qmbceq\ w.r.t. first-order swap structures with standard equality)} \label{comp-QmbCeq} Let $\Gamma \cup \{\varphi\} \subseteq For(\Theta_\approx)$.  If $\Gamma \models_{\qmbceq} \varphi$ then $\Gamma\vdash_{\qmbceq} \varphi$.
\end{corollary}

\section{Completeness of \qmbceq\ w.r.t. structures with standard equality over $\matM_5$} \label{compM5eq}

In this section the adequacy of \qmbc\ w.r.t. first-order swap structures stated in Theorem~\ref{adeq-qmbcM5} will be extended to \qmbceq. In order to do this, some definitions taken from  Section~\ref{compM5} will be adapted to \qmbceq, by  following the approach in~\cite[Section~7.7]{CC16} with small modifications. In particular, \cite[Definition~7.7.3]{CC16} will be slightly adapted as follows:

\begin{definition} \label{sem2Qmbceq}
An {\em interpretation} for \qmbceq\ over a signature $\Theta_\approx$ is a pair $\langle \mathsf{A}, \rho\rangle$ such that $\mathsf{A}$ is a standard Tarskian first-order structure with standard equality over $\Theta_\approx$\footnote{That is, $\mathsf{A}$ is a standard Tarskian first-order structure over signature $\Theta_\approx$, as considered in Section~\ref{compM5}, in which the equality predicate $\approx$ is interpreted as the identity: $\approx^\mathsf{A} \defin \{(a,a) \ : \ a \in U\}$.} and $\rho$ is a bivaluation for \qmbc\ over  $\mathsf{A}$. The {\em consequence relation $\models_{\qmbceq}^2$ of \qmbceq\ w.r.t. interpretations} is given by: $\Gamma\models_{\qmbceq}^2 \varphi$ if, for every interpretation $\langle \mathsf{A}, \rho\rangle$ for \qmbceq: $\rho(\widehat{\mu}(\gamma))=1$ for every $\gamma \in \Gamma$ and every $\mu$ implies that $\rho(\widehat{\mu}(\varphi))=1$ for every $\mu$.
\end{definition}

\begin{remark} \label{obsvaleq} \ \\
(1) In~\cite[Definition~7.7.3]{CC16} it was introduced the notion of \qmbceq-valuations, which are bivaluations for \qmbc\ over  standard Tarskian structures $\mathsf{A}$ over $\Theta_\approx$ satisfying  for $\approx$, instead of~(\valpred), the following clauses:\\[2mm]
{\bf (\valeq1)} \ $\rho(t_1 \approx t_2) = 1$ \  iff \ 
$\termvalue{t_1}^{\widehat{\mathsf{A}}} = \termvalue{t_2}^{\widehat{\mathsf{A}}}$ for every $t_1,t_2 \in CTer(\Theta_U^\approx)$;\\[1mm]
{\bf (\valeq2)} \ $\rho(\bar a \approx \bar b)=1$  implies $\rho(\alpha[x,y/\bar a,\bar b]) = \rho(\alpha[x\wr y][x,y/\bar a,\bar b])$ for every $a,b\in U$, if
$y$ is a variable free for $x$ in $\alpha$.\\[2mm]
It is easy to see that~(\valeq2) is derivable from~(\valeq1). Indeed, suppose that $\rho(\bar a \approx \bar b)=1$. Hence $a = \termvalue{\bar a}^{\widehat{\mathsf{A}}} = \termvalue{\bar b}^{\widehat{\mathsf{A}}} = b$, by  clause~(\valeq1). But then $\bar a = \bar b$, which implies that $\alpha[x,y/\bar a,\bar b] = \alpha[x\wr y][x,y/\bar a,\bar b]$. Thus, 
$\rho(\alpha[x,y/\bar a,\bar b]) = \rho(\alpha[x\wr y][x,y/\bar a,\bar b])$, showing that $\rho$ also satisfies~(\valeq2).\\[1mm]
(2) Let $\langle \mathsf{A}, \rho\rangle$ be an interpretation for \qmbceq\ as in Definition~\ref{sem2Qmbceq}. Then, by~(\valpred) applied to $\approx$ (recall Definition~\ref{bivalqmbc}), and by the fact that $\approx^\mathsf{A} = \{(a,a) \ : \ a \in U\}$, it follows that $\rho$ satisfies clause~(\valeq1), hence it also satisfies~(\valeq2), by item~(1) above. This means that $\langle \mathsf{A}, \rho\rangle$ is an interpretation for \qmbceq\ in the sense of~\cite[Definition~7.7.3]{CC16}. Conversely, if  $\langle \mathsf{A}, \rho\rangle$ is  an interpretation for \qmbceq\ in the sense of~\cite[Definition~7.7.3]{CC16}  let  $\mathsf{A}'$ be the standard  Tarskian structure over $\Theta_\approx$ obtained from $\mathsf{A}$ by setting  $\approx^{\mathsf{A}'} \defin \{(a,a) \ : \ a \in U\}$.  Hence  $\langle \mathsf{A}', \rho\rangle$ is an interpretation for \qmbceq\ as in Definition~\ref{sem2Qmbceq}, since $\rho$ satisfies~(\valeq1) and so it satisfies~(\valpred) applied to $\approx$. This shows that our presentation is equivalent to that of~\cite{CC16}. 
\end{remark}

\begin{theorem} {\em (Adequacy of \qmbceq\ w.r.t. interpretations, \cite[Theorem~7.7.5]{CC16})} \label{comple2Qmbceq} If $\Gamma \cup \{\varphi\} \subseteq For(\Theta_\approx)$ then:  $\Gamma\vdash_{\qmbceq} \varphi$ \ iff \ $\Gamma\models_{\qmbceq}^2 \varphi$.\footnote{As in the case of \qmbc, in~\cite[Theorem~7.7.5]{CC16} it was obtained adequacy of \qmbceq\ w.r.t. interpretations, but only for sentences. Once again, adequacy for general formulas follows from that result by using universal closure.}
\end{theorem}

\noindent Theorem~\ref{val-bival-qmbC} can be easily extended to \qmbceq:

\begin{theorem}  \label{val-bival-qmbCeq}
Let  $\mathcal{I}=\langle \mathsf{A}, \rho\rangle$ be an interpretation for \qmbceq\ over a signature $\Theta_\approx$. Then, it induces a first-order structure with standard equality $\mathfrak{A}_\mathcal{I}$ over $\matM_5$ and $\Theta_\approx$, and a \qmbceq-valuation $v^\rho_\approx$ over  $\mathfrak{A}_\mathcal{I}$ and $\matM_5$  given by $v^\rho_\approx(\alpha) \defin (\rho(\alpha),\rho(\neg\alpha),\rho(\cons\alpha))$ such that: $\rho(\alpha)=1$ iff $v^\rho_\approx(\alpha) \in {\rm D}$, for every sentence $\alpha \in Sen(\Theta_U^\approx)$.
\end{theorem}
\begin{proof} Let $\mathcal{I}=\langle \mathsf{A}, \rho\rangle$ be an interpretation for \qmbceq\ over signature $\Theta_\approx$. Consider the first-order structure $\mathfrak{A}_\mathcal{I}$ over $\matM_5$ and $\Theta_\approx$ obtained from $\mathsf{A}$ as in the proof of Theorem~\ref{val-bival-qmbC}. In particular, $I_{\mathfrak{A}_\mathcal{I}}(\approx)(a_1,a_2)= v^\rho_\approx(\bar{a}_1 \approx\bar{a}_n)$,
where $v^\rho_\approx:Sen(\Theta_U^\approx) \to |\matM_5|$ is given by $v^\rho_\approx(\alpha) = (\rho(\alpha),\rho(\neg\alpha),\rho(\cons\alpha))$, for every  $\alpha \in Sen(\Theta_U^\approx)$. It is immediate to see that $\rho(\alpha)=1$ iff $v^\rho_\approx(\alpha) \in {\rm D}$, for every  $\alpha \in Sen(\Theta_U^\approx)$. From this,
$I_{\mathfrak{A}_\mathcal{I}}(\approx)(a_1,a_2)= v^\rho_\approx(\bar{a}_1 \approx\bar{a}_n) \in {\rm D}$ iff  $\rho(\bar{a}_1 \approx\bar{a}_2)= 1$ iff $a_1=a_2$,
since $\rho$ satisfies clause~(\valeq1), by Remark~\ref{obsvaleq}. This shows that $\mathfrak{A}_\mathcal{I}$ is a first-order structure with standard equality  over $\matM_5$ and $\Theta_\approx$.

In order to see that $v^\rho_\approx$ is a \qmbceq-valuation $v^\rho_\approx$ over  $\mathfrak{A}_\mathcal{I}$ and $\matM_5$ observe that $v^\rho_\approx$  satisfies clause~(i) of Definition~\ref{val} for every predicate symbol in $\Theta$. 
Concerning the equality predicate $\approx$ it is easy to prove that that, by the axioms of equality, the properties of  bivaluations for \qmbc, and the fact that   $a=\termvalue{\bar a}^{\widehat{\mathfrak{A}_\mathcal{I}}}=\termvalue{\bar a}^{\widehat{\mathsf{A}}}$ for every $a \in U$ (in particular, for $a=\termvalue{t}^{\widehat{\mathfrak{A}_\mathcal{I}}}$, for $t \in CTer(\Theta_U^\approx)$),
$v^\rho_\approx(t_1 \approx t_2)  =  v^\rho_\approx\left(\overline{\termvalue{t_1}^{\widehat{\mathfrak{A}_\mathcal{I}}}} \approx \overline{\termvalue{t_2}^{\widehat{\mathfrak{A}_\mathcal{I}}}}\right) = I_{\mathfrak{A}_\mathcal{I}}(\approx)\left(\termvalue{t_1}^{\widehat{\mathfrak{A}_\mathcal{I}}},\termvalue{t_2}^{\widehat{\mathfrak{A}_\mathcal{I}}}\right)$. This shows that $v^\rho_\approx$ also satisfies clause~(i) of Definition~\ref{val} for the equality predicate $\approx$. By the proof of Theorem~\ref{val-bival-qmbC}, it follows that $v^\rho_\approx$ satisfies the other clauses of Definition~\ref{val} for \qmbc-valuation over  $\mathfrak{A}_\mathcal{I}$ and $\matM_5$.
It remains to prove that $v^\rho_\approx$ is a \qmbceq-valuation over  $\mathfrak{A}_\mathcal{I}$ and $\matM_5$, that is, $v^\rho_\approx(\widehat{\mu}(\alpha)) \in {\rm D}$ for every instance $\alpha$ of  axiom~({\bf AxEq2}) and every assignment $\mu$. But this is easy to prove,  by the properties of $\rho$ and by an argument similar to that presented in Remark~\ref{obsvaleq}(1). This concludes the proof.
\end{proof}

\noindent As an immediate consequence of  Theorems~\ref{sound-Qmbceq-swap}, \ref{comple2Qmbceq} and~\ref{val-bival-qmbCeq}:

\begin{theorem} {\em (Adequacy of \qmbceq\ w.r.t. first-order structures with standard equality over $\matM_5$)} \label{adeq-qmbcM5eq}
For every set $\Gamma \cup \{\varphi\} \subseteq For(\Theta_\approx)$:  $\Gamma\vdash_{\qmbceq} \varphi$ \ iff \ $\Gamma\models_{(\mathfrak{A}, \matM_5)}^\approx \varphi$ for every  structure $\mathfrak{A}$ with  standard equality over $\Theta_\approx$ and $\matM_5$.
\end{theorem}

\section{First order twist structures  based on the logic $\lfium_\cons$} \label{Qtwist}

The generalization of  swap structures semantics to other quantified \lfis, defined as axiomatic extensions of \qmbc, can be easily obtained. Indeed, by analyzing the  swap structures semantics for axiomatic extensions of \mbc\ given in~\cite[Section~6.5]{CC16} (see also~\cite{CFG18}), as well as the first-order version of such extensions proposed in~\cite[Section~7.8]{CC16}, it is immediate how to obtain first-order swap structures for all these logics. 
Thus, it is immediate to define \textbf{QmbCciw}, \textbf{QmbCci}, \textbf{QbC} and \textbf{QCi}, the quantified version of \textbf{mbCciw}, \textbf{mbCci}, \textbf{bC} and \textbf{Ci}, respectively, as well as the corresponding extensions of them by adding the standard equality. All these logics are characterized by means of first-order structures defined over 3-valued swap structures.\footnote{As proved by Avron in~\cite{avr:07}, the extension  \textbf{mbCcl} of \mbc\ by adding da Costa's axiom (cl): $\neg(\varphi \wedge \neg\varphi) \to \cons\varphi$ cannot be characterized by a single finite Nmatrix. This negative result also holds for  \textbf{Cila}, the presentation of a Costa's system $C_1$ in the language of \lfis, hence it holds for $C_1$ itself. Thus, these systems lies ouside the scope of the present framework. A  discussion about this question can be found in~\cite[Section~6.5]{CC16}. On the other hand, \textbf{Cila} is not algebraizable in the sense of Blok and Pigozzi (consult, for instance, \cite[Section~3.12]{CM02}). As a consequence of this, none of the logics \mbc, \textbf{mbCciw}, \textbf{mbCci}, \textbf{bC} and \textbf{Ci} is algebraizable.}

Instead of analyzing in this section these axiomatic extensions of \qmbc, together with the corresponding swap structures semantics (a straightforward exercise),  the first-order version of a quite interesting axiomatic extension of \mbc, the logic $\lfium_\cons$,  will be analyzed with full detail. This logic can be semantically characterized by a 3-valued logical matrix called \lfiuml, which is equivalent (up to language) to several 3-valued paraconsistent logics such as the well-known da Costa-D'Ottaviano logic \dacdot\ and Carnielli-Marcos-de Amo logic \lfium. From this, it follows that  $\lfium_\cons$ is algebraizable in the sense of Blok and Pigozzi (see~\cite[Chapter~4]{CC16} for a discussion about this logic). The interesting point is that, as proved in~\cite{CFG18}, the swap structures for $\lfium_\cons$ turn out to be deterministic, thus becoming  twist structures, which represent the algebraic semantics for $\lfium_\cons$.\footnote{Twist structures were independently introduced by Fidel~\cite{fid:78} and Vakarelov~\cite{vaka:77}. However, as observed by Cignoli in~\cite{cin:86}, the basic algebraic ideas underlying twist structures were firstly introduced by Kalman in~\cite{kal:58}.} Because of this, the first-order structures for the first-order extension of $\lfium_\cons$ presented here are based on twist structures.

\begin{definition} (\cite[Definition~4.4.39]{CC16}) \label{LFI1-def}
Let $\matM_{LFI1} = \langle  \aptz, \{1,\frac{1}{2}\}\rangle$ be the logical matrix such that  \aptz\ is the  3-valued algebra  over $\Sigma$ with domain $M= \{1,\frac{1}{2}, 0\}$, where the operators are defined as follows:

{\small
\begin{center}
\begin{tabular}{|c||c|c|c|}
\hline
$\land $ & 1 & $\frac12$ & 0 \\ \hline \hline
1 & 1 & $\frac12$ & 0 \\ \hline
$\frac12$ & $\frac{1}{2}$ & $\frac12$ & 0 \\ \hline
0 & 0 & 0 & 0 \\ \hline
\end{tabular}
\hspace{0.1cm}
 \begin{tabular}{|c||c|c|c|}
\hline
$\lor $ & 1 & $\frac12$ & 0 \\ \hline \hline
1 & 1 & 1 & 1 \\ \hline
$\frac12$ & 1 & $\frac{1}{2}$ & $\frac{1}{2}$ \\ \hline
0 & 1 & $\frac{1}{2}$ & 0 \\ \hline
\end{tabular}
\hspace{0.1cm}
\begin{tabular}{|c||c|c|c|}
\hline
$\imp$ & 1  & $\frac{1}{2}$  & 0 \\
 \hline \hline
     1    & 1  & $\frac{1}{2}$  & 0   \\ \hline
     $\frac{1}{2}$  & 1  & $\frac{1}{2}$  & 0   \\ \hline
     0    & 1  & 1  & 1   \\ \hline
\end{tabular}
\hspace{0.1cm}
\begin{tabular}{|c||c|} \hline
$\quad$ & $\wneg$ \\
 \hline \hline
     1   & 0    \\ \hline
     $\frac{1}{2}$ & $\frac{1}{2}$  \\ \hline
     0   & 1    \\ \hline
\end{tabular}
\hspace{0.1cm}
\begin{tabular}{|c||c|} \hline
$\quad$ & $\cons$ \\
 \hline \hline
     1   & 1    \\ \hline
     $\frac{1}{2}$ & 0  \\ \hline
     0   & 1    \\ \hline
\end{tabular}
\end{center}
}

\ 

\noindent The logic  associated to the logical matrix $\matM_{LFI1}$ is called \lfiuml.
\end{definition}

\noindent Recall that, by definition, the consequence relation $\models_{\lfiuml}$ of \lfiuml\ is given as follows: for every $\Gamma\cup\{\alpha\} \subseteq {\cal L}_{\Sigma}$, $\Gamma\models_{\lfiuml} \alpha$ iff, for every homomorphism $v:{\cal L}_{\Sigma} \to  \aptz$ of algebras over $\Sigma$, if $v[\Gamma] \subseteq \{1,\frac{1}{2}\}$ then $v(\alpha) \in \{1,\frac{1}{2}\}$.

A sound and complete Hilbert calculus for \lfiuml, called $\lfium_\circ$, was introduced in~\cite{CC16}. This calculus is an axiomatic extension  of \mbc.

\begin{definition} (\cite[Definition~4.4.41]{CC16}) \label{LPT0-def} The Hilbert calculus  $\lfium_\circ$  over $\Sigma$ is obtained from  \mbc\ by adding the following axioms:\\[1mm]
$\begin{array}{ll}
(\axci) & \neg\cons\alpha \to (\alpha \wedge \neg \alpha)\\[1mm]
(\axdneg) & \wneg\wneg\alpha \leftrightarrow \alpha\\[1mm]
(\axnegou) & \wneg(\alpha \lor \beta) \leftrightarrow (\wneg \alpha \land \wneg\beta)\\[1mm]
(\axnege) & \wneg(\alpha \land \beta) \leftrightarrow (\wneg \alpha \lor \wneg\beta)\\[1mm]
(\axnegimp) & \wneg(\alpha \imp \beta) \leftrightarrow(\alpha \land \wneg\beta)\\[1mm]
\end{array} $
\end{definition}

\begin{theorem} {\em (\cite[Theorem~4.4.45]{CC16})} \label{adeqMPT0}
The logic $\lfium_\circ$  is sound and complete  w.r.t. the matrix semantics of \lfiuml: $\Gamma\vdash_{\lfium_\circ} \alpha$ \ iff \ $\Gamma\models_{\lfiuml} \alpha$
\end{theorem}

\begin{remark} \label{cons-defin}
By Propositions~3.1.10 and~3.2.3 in~\cite{CC16}, and given that $\lfium_\circ$ contains axiom (\axci), it follows that $\vdash_{\lfium_\circ} \circ\alpha \leftrightarrow \sneg(\alpha \wedge \neg\alpha)$ (here, \sneg\ is the classical negation defined as in Remark~\ref{univ-Henk}). On the other hand, taking into account that $\vdash_{\mbc} (\alpha \wedge \neg\alpha) \to \neg\cons\alpha$, it follows from (\axci) that   $\vdash_{\lfium_\circ} \neg\cons\alpha \leftrightarrow (\alpha \wedge \neg\alpha)$.
\end{remark}

\noindent
Since $\lfium_\circ$ is an axiomatic extension of \mbc, its first-order extension $\qlfium_\circ$ can be defined as an axiomatic extension of \qmbc,\footnote{As we shall see in Remark~\ref{Ax14}, axiom {\bf (Ax14)} will be redundant.} hence the  semantics of first-order swap structures for \qmbc\ given in the previous sections can be adapted to  $\qlfium_\circ$, obtaining so a semantics based on first-order twist structures. Indeed, as shown in~\cite{CFG18}, each multioperation in the corresponding swap structures for $\lfium_\circ$ is deterministic, and so these swap structures are twist structures (which are ordinary algebras presented in a particular form).

\begin{definition}
Let $\mA=\langle A, \wedge, \vee, \to,0,1 \rangle$ be a Boolean algebra. The {\em twist domain} generated by \mA\ is the set $\tA=\{(z_1,z_2) \in A  \times A \ : \ z_1 \vee z_2 = 1\}$. 
\end{definition}

\begin{definition}  (\cite[Definition~9.2]{CFG18})  \label{defKlfi1}
Let \mA\ be a Boolean algebra. The  {\em twist structure for $\lfium_\cons$ over \mA} is the algebra  $\tmA=\langle \tA, \tilde{\wedge}, \tilde{\vee}, \tilde{\imp},\tilde{\wneg},\tilde{\cons} \rangle$ over $\Sigma$  such that the operations are defined as follows, for every $(z_1,z_2),(w_1,w_2) \in \tA$:\\[1mm]
(i) $(z_1,z_2)\,\tilde{\wedge}\, (w_1,w_2) = (z_1 \wedge w_1,z_2 \vee w_2)$;\\
(ii)  $(z_1,z_2)\,\tilde{\vee}\, (w_1,w_2) = (z_1 \vee w_1,z_2 \wedge w_2)$;\\
(iii)  $(z_1,z_2)\,\tilde{\imp}\, (w_1,w_2) = (z_1 \imp w_1,z_1 \wedge w_2)$;\\
(iv) $\tilde{\neg} (z_1,z_2) = (z_2,z_1)$;\\
(v) $\tilde{\circ}(z_1,z_2)=(\sneg(z_1 \land z_2),z_1 \land z_2)$.\footnote{Here, $\sneg$ denotes the Boolean complement $\sneg x = x \to 0$.}
\end{definition}

\noindent
The intuitive meaning of a snapshot $(z_1,z_2)$ in $\tA$ is that $z_1$ represents a value, in a given Boolean algebra, for the evidence for $\varphi$, while $z_2$ represents a value for the evidence against $\varphi$ (or a value for the evidence for $\neg\varphi$).

\begin{definition}  \label{def-semKlfi1}
The logical matrix associated to the twist structure \tmA\ is $\matA=\langle \tmA,D_\mA\rangle$ where $D_\mA = \{(z_1,z_2) \in \tA \ : \  z_1=1\} = \{(1,a) \ : \  a \in A\}$. The consequence relation associated to \matA\ will be denoted by $\models_{\tmA}$, namely: $\Gamma \models_{\tmA} \alpha$ iff, for every homomorphism $h:{\cal L}_{\Sigma} \to  \tA$ of algebras over $\Sigma$, if $h(\gamma) \in D_\mA$ for every $\gamma \in \Gamma$ then  $h(\alpha) \in D_\mA$. Let $\mathbb{M}_{LFI1}$ be the class of logical matrices \matA, for any Boolean algebra \mA. The {\em twist consequence relation} for $\lfium_\cons$ is the consequence relation $\models_{\mathbb{M}_{LFI1}}$ associated to $\mathbb{M}_{LFI1}$, namely: $\Gamma \models_{\mathbb{M}_{LFI1}} \alpha$ iff $\Gamma \models_{\tmA} \alpha$ for every Boolean algebra \mA.
\end{definition}

\begin{remark}  \label{twistA2}
In~\cite[Theorem~9.6]{CFG18} it was shown that $\lfium_\cons$ is sound and complete w.r.t. twist structures  semantics, namely: $\Gamma \vdash_{\lfium_\cons} \alpha$ iff $\Gamma\models_{\mathbb{M}_{LFI1}} \alpha$, for every set of formulas $\Gamma \cup\{\alpha\}$. On the other hand, if $\mA_2$  is the two-element Boolean algebra with domain $\{0,1\}$ then $\mathcal{T}_{\mA_2}$ consists of three elements: $(1,0)$, $(1,1)$ and $(0,1)$. By identifying these elements with $1$, $\frac{1}{2}$ and $0$, respectively, then $\mathcal{T}_{\mA_2}$ coincides with the 3-valued algebra \aptz\ underlying the matrix $\matM_{LFI1}$ (recall Definition~\ref{LFI1-def}). Moreover, $\mathcal{MT}_{\mA_2}$ coincides with $\matM_{LFI1}$. Taking into consideration Theorem~\ref{adeqMPT0}, this situation is analogous to the semantical characterization of  \mbc\ w.r.t. the 5-element swap structure over  $\mA_2$: it is enough to consider the structure induced by $\mA_2$ in order to characterize the logic.
\end{remark}

\noindent
A first-order version of $\lfium_\circ$, which will be called $\qlfium_\circ$, can be easily defined from \qmbc.
 
\begin{definition} Let $\Theta$ be a first-order signature. The logic $\qlfium_\circ$ is obtained
from \qmbc\  by deleting axiom {\bf (Ax14)} and by adding axioms {\bf (\axci)}, {\bf 
(\axdneg)}, {\bf (\axnegou)}, {\bf (\axnege)} and {\bf (\axnegimp)} from $\lfium_\circ$, plus the following:\\

$\begin{array}{ll}
{\bf (Ax\neg\exists)} & \neg\exists x\varphi\leftrightarrow \forall x \neg\varphi\\[3mm]
{\bf (Ax\neg\forall)} & \neg\forall x\varphi\leftrightarrow \exists x  \neg\varphi
\end{array}$
\end{definition}

\begin{remark} \label{Ax14} Observe that $\qlfium_\circ$  can be alternatively defined as the Hilbert calculus obtained from  $\lfium_\cons$ by adding axioms {\bf (Ax12)} and {\bf (Ax13)} from Definition~\ref{defqmbc}, 
${\bf (Ax\neg\exists)}$ and ${\bf (Ax\neg\forall)}$ above, and the inference rules ${\bf (\exists\mbox{\bf -In})}$ and ${\bf (\forall\mbox{\bf -In})}$ from Definition~\ref{defqmbc}.
The fact that axiom {\bf (Ax14)} is no longer required is justified by the fact that it can now be derived from the other axioms. This can be proved easily after obtaining  the completeness of  $\qlfium_\circ$  w.r.t. twist structures semantics, since  axiom {\bf (Ax14)} is valid w.r.t. that semantics.
\end{remark}

\noindent
The consequence relation of $\qlfium_\circ$ will be denoted by $\vdash_{\qlfium_\circ}$. Since  $\qlfium_\circ$ does not add inference rules to \qmbc, it satisfies a deduction meta-theorem (DMT) analogous to \qmbc\ (see Theorem~\ref{teoded:teo}).

Now, the swap structures semantics for \qmbc\ can be adapted to $\qlfium_\circ$, taking into account that the swap structures for $\lfium_\circ$ are exactly the twist structures  introduced in Definition~\ref{defKlfi1}. This leads us to the following definition:

\begin{definition} \label{tstru} let \mA\ be a complete Boolean algebra. Let \matA\ be the logical matrix associated to the twist structure \tmA\  for $\lfium_\circ$, and let $\Theta$ be a first-order signature. A  (first-order) {\em structure} over \matA\  and $\Theta$, or a {\em $\qlfium_\circ$-structure} over $\Theta$, is a pair $\mathfrak{A} = \langle U, I_{\mathfrak{A}} \rangle$ as in Definition~\ref{stru}, but now  $I_\mathfrak{A}(P)$ is a function from $U^n$ to $\tA$, for each predicate symbol $P$ of arity $n$.
\end{definition}

\begin{Not} {\em As it was done with \qmbc, $c^\mathfrak{A}$, $f^\mathfrak{A}$ and $P^\mathfrak{A}$  will denote the interpretation of an individual constant symbol $c$, a function symbol $f$ and a predicate symbol $P$, respectively.}
\end{Not}

\noindent The notion of  assignment over a $\qlfium_\circ$-structure is as in Definition~\ref{assign}. The notion of interpretation  $\termvalue{t}^\mathfrak{A}_\mu$ of a term $t$ given  a structure $\mathfrak{A}$ and an assignment $\mu$ is identical to the one described in  Definition~\ref{term}. Given $\mathfrak{A}$, the structure $\widehat{\mathfrak{A}} = \langle U, I_{\widehat{\mathfrak{A}}} \rangle$ over $\Theta_U$  is defined analogously to the case of \qmbc\ (recall Definition~\ref{extA}).

\begin{Not} \label{projz}
{\em By adapting Notation~\ref{Pi1}, if $z \in \tA$ then  $(z)_1$ and $(z)_2$, or simply $z_1$ and $z_2$, will denote the first and second coordinates of $z$, respectively. }
\end{Not}

\noindent  As it was discussed after Definition~\ref{mu}, in order to obtain a single denotation (truth-value) for a formula in \qmbc, a  given interpretation and an assignment are not enough:  valuations are necessary in order to choose a unique denotation, in case the formula is complex (that is, if it contains connectives or quantifiers). The case of $\qlfium_\circ$ is different, since twist structures are deterministic (that is, they are ordinary algebras). This being so, from a given denotation for the atomic formulas, the denotation for complex formulas is uniquely determined fom the denotation of its components, which is in line with the  traditional approach to first-order algebraic logic originated by Mostowski. Because of this, valuations over structures  are no longer necessary for  $\qlfium_\circ$, and a structure $\mathfrak{A}$ will assign a single denotation (truth-value), denoted by $\termvalue{\varphi}^{\mathfrak{A}}$,  to each sentence $\varphi$. Thid lead us to the following definition:

\begin{definition} ($\qlfium_\circ$ interpretation maps)~\label{val1}
Let \mA\ be a complete Boolean algebra, and let $\mathfrak{A}$ be  a structure over \matA\ and $\Theta$. The {\em interpretation map} for $\qlfium_\circ$ over $\mathfrak{A} $ and  \matA\ is a function $\termvalue{\cdot}^{\mathfrak{A}}:Sen(\Theta_U) \to \tA$ satisfying the following clauses (using Notation~\ref{projz} in clauses~(iv) and~(v)): \\[1mm]
(i) $\termvalue{P(t_1,\ldots,t_n)}^{\mathfrak{A}} = P^{\mathfrak{A}}(\termvalue{t_1}^{\widehat{\mathfrak{A}}},\ldots,\termvalue{t_n}^{\widehat{\mathfrak{A}}}) $, if $P(t_1,\ldots,t_n)$ is atomic; \\[1mm]
(ii) $\termvalue{\#\varphi}^{\mathfrak{A}} = \tilde{\#} \termvalue{\varphi}^{\mathfrak{A}}$, for every $\#\in \{\neg, \cons\}$;\\[1mm]
(iii)  $\termvalue{\varphi \# \psi}^{\mathfrak{A}} = \termvalue{\varphi}^{\mathfrak{A}} \, \tilde{\#} \, \termvalue{\psi}^{\mathfrak{A}}$, for every $\#\in \{\wedge,\vee, \to\}$;\\[1mm]
(iv) $\termvalue{\forall x\varphi}^{\mathfrak{A}} = \big(\bigwedge_{a \in U} (\termvalue{\varphi[x/\bar{a}]}^{\mathfrak{A}})_1,\bigvee_{a \in U} (\termvalue{\varphi[x/\bar{a}]}^{\mathfrak{A}})_2\big)$;\\[1mm]
(v) $\termvalue{\exists x\varphi}^{\mathfrak{A}} = \big(\bigvee_{a \in U} (\termvalue{\varphi[x/\bar{a}]}^{\mathfrak{A}})_1,\bigwedge_{a \in U} (\termvalue{\varphi[x/\bar{a}]}^{\mathfrak{A}})_2\big)$.
\end{definition}

\noindent
The definition of the interpretation of the quantifiers  in $\qlfium_\circ$ is coherent with the fact that  $\tmA$\ (ordered by: $z \leq w$ iff $z_1 \leq w_1$ and $z_2 \geq w_2$) is a complete lattice (since $\mA$ is a complete Boolean algebra), in which
\begin{itemize}
\item[] $\bigwedge_{i \in I} z_i = \big(\bigwedge_{i \in I} (z_i)_1,\bigvee_{i \in I} (z_i)_2\big)$, and
\item[] $\bigvee_{i \in I} z_i = \big(\bigvee_{i \in I} (z_i)_1,\bigwedge_{i \in I} (z_i)_2\big)$ 
\end{itemize}
for every family $(z_i)_{i \in I}$ in $\tA$. Note that $1_{\tmA}=_{def}(1,0)$ and $0_{\tmA}=_{def}(0,1)$ are the top and bottom elements of \tmA, respectively.

\begin{definition}  Let \mA\ be a complete Boolean algebra,  $\mathfrak{A}$  a structure over \matA\ and $\Theta$, and let $\mu$  be  an assignment over $\mathfrak{A}$. The {\em extended interpretation map} $\termvalue{\cdot}^{\mathfrak{A}}_{\mu}:For(\Theta_U) \to \tA$  is given by $\termvalue{\varphi}^{\mathfrak{A}}_{\mu} = \termvalue{\widehat{\mu}(\varphi)}^{\mathfrak{A}}$.
\end{definition}

\begin{definition} \label{consrel0} Let \mA\ be a complete Boolean algebra, and let $\mathfrak{A}$ be  a structure over \matA\ and $\Theta$.  Given a set of formulas $\Gamma \cup \{\varphi\} \subseteq For(\Theta_U)$, $\varphi$ is said to be  a {\em semantical consequence of $\Gamma$ w.r.t. $(\mathfrak{A}, \matA)$}, denoted by $\Gamma\models_{(\mathfrak{A}, \matA)}\varphi$, if the following holds:  if $\termvalue{\gamma}^{\mathfrak{A}}_{\mu} \in D$,  for every formula $\gamma \in \Gamma$ and every assignment $\mu$, then $\termvalue{\varphi}^{\mathfrak{A}}_{\mu} \in D$,  for every assignment $\mu$.
\end{definition}

\begin{definition} (Semantical consequence relation in $\qlfium_\circ$ w.r.t. twist structures)  Let $\Gamma \cup \{\varphi\} \subseteq For(\Theta)$ be a set of formulas. Then  $\varphi$ is said to be  a {\em semantical consequence of $\Gamma$ in $\qlfium_\circ$ w.r.t. first-order twist structures}, denoted by $\Gamma\models_{\qlfium_\circ}\varphi$, if $\Gamma\models_{(\mathfrak{A}, \matA)}\varphi$ for every $(\mathfrak{A}, \matA)$.
\end{definition}

\noindent 
The  soundness of $\qlfium_\circ$ w.r.t. first-order twist structures semantics can be easily obtained. The proof is analogous but much easier than the proof for \qmbc\ given in Theorem~\ref{sound-Qmbc-swap}, given that valuations are no longer necessary.



\begin{proposition}~\label{paxioms1} \ \\
(i) $\alpha, \, \alpha \to \beta \models_{\qlfium_\circ} \beta$; \\
(ii) $\alpha \to \beta \models_{\qlfium_\circ} \exists x \alpha \to \beta$, if $x$ is not free in $\beta$;\\ 
(iii) $\alpha \to \beta \models_{\qlfium_\circ} \alpha \to \forall x \beta$, if $x$ is not free in $\alpha$;\\
(iv) $ \models_{\qlfium_\circ} \forall x \alpha \to \alpha[x/t]$, if $t$ is a term free for $x$ in $\alpha$;\\
(v) $\models_{\qlfium_\circ} \alpha[x/t] \to \exists x \alpha$, if $t$ is a term free for $x$ in $\alpha$;\\
(vi) $\models_{\qlfium_\circ} \neg\exists x\varphi\leftrightarrow \forall x \neg\varphi$;\\
(vii) $\models_{\qlfium_\circ} \neg\forall x\varphi\leftrightarrow \exists x  \neg\varphi$.
\end{proposition}

\begin{proof} 
The proofs for items~(i)-(v) are easily obtained by adapting the corresponding proofs for \qmbc\ given in Proposition~\ref{paxioms}. Items~(vi) and~(vii) follow easily from the definitions. 
\end{proof}

\begin{corollary}~\label{psound1}
Let \mA\ be a complete Boolean algebra, and let  $\mathfrak{A}$ be a structure over \matA\ and $\Theta$. Then:\\
(1) If $\alpha$ is an instance of an axiom schema of $\qlfium_\circ$ then $\termvalue{\alpha}^{\mathfrak{A}}_{\mu} \in D_\mA$, for every assignment $\mu$.\\
(2) If $\alpha$ and $\beta$ are formulas such that $\termvalue{\alpha}^{\mathfrak{A}}_{\mu} \in D_\mA$ and $\termvalue{\alpha \to \beta}^{\mathfrak{A}}_{\mu} \in D_\mA$ for every assignment $\mu$, then $\termvalue{\beta}^{\mathfrak{A}}_{\mu} \in D_\mA$ for every $\mu$.\\
(3) If $\alpha$ and $\beta$ are formulas such that $\termvalue{\alpha \to \beta}^{\mathfrak{A}}_{\mu} \in D_\mA$ for every assignment $\mu$, and if $x$ does not occur free in $\beta$, then $\termvalue{\exists x \alpha \to \beta}^{\mathfrak{A}}_{\mu} \in D_\mA$ for every $\mu$.\\
(4) If $\alpha$ and $\beta$ are formulas such that $\termvalue{\alpha \to \beta}^{\mathfrak{A}}_{\mu} \in D_\mA$, for every assignment $\mu$, and if $x$ does not occur free in $\alpha$, then $\termvalue{\alpha \to \forall x\beta}^{\mathfrak{A}}_{\mu} \in D_\mA$ for every $\mu$.
\end{corollary}

\begin{theorem} {\em (Soundness of $\qlfium_\circ$ w.r.t. first-order twist structures)} \label{sound-Qlfi1-twist}
For every set $\Gamma \cup\{ \varphi\}  \subseteq For(\Theta)$:  if $\Gamma \vdash_{\qlfium_\circ} \varphi$, then  $\Gamma \models_{\qlfium_\circ} \varphi$.
\end{theorem}

\begin{proof}

By induction on the length $n$ of a derivation of $\varphi$ from $\Gamma$ in $\qlfium_\circ$, taking into account Corollary~\ref{psound1}.
\end{proof}

\section{Completeness of $\qlfium_\circ$ w.r.t. twist structures} \label{complete1}

Now, the  completeness of $\qlfium_\circ$ w.r.t. first-order twist structures semantics will be proved, by adapting the completeness proof for \qmbc\ given in Section~\ref{complete}.

The notion of $C$-\emph{Henkin theory} in $\qlfium_\circ$ is analogous to the one for \qmbc\ given in Definition~\ref{Henkin}. The consequence relation $\vdash^{C}_{\qlfium_\circ}$ is the consequence relation of $\qlfium_\circ$ over the signature $\Theta_{C}$. The following result can be proved by adapting the proof for \qmbc\ (see Proposition~\ref{saturated}):

\begin{proposition} \label{saturated1}
Let  $\Gamma \cup \{\varphi\} \subseteq Sent(\Theta)$  such that $\Gamma \nvdash_{\qlfium_\circ} \varphi$. Then, there exists a set of sentences $\Delta \subseteq Sent(\Theta)$ which is maximally non-trivial with respect to $\varphi$ in $\qlfium_\circ$ (assuming that $\vdash_{\qlfium_\circ}$ is restricted to sentences) and such that $\Gamma \subseteq \Delta$.
\end{proposition}

\begin{definition} Let $\Delta \subseteq Sen(\Theta)$ be a non-trivial theory in $\qlfium_\circ$. Let ${\equiv_{\Delta}^1} \subseteq$ \sent$^{2}$ be the relation in \sent\ given by:  $\alpha \equiv_{\Delta}^1 \beta$ iff $\Delta \vdash_{\qlfium_\circ} \alpha \leftrightarrow\beta$.
\end{definition}

\noindent
As in the case of \qmbc, $\equiv_{\Delta}^1$ is an equivalence relation.  The  equivalence class of a sentence $\alpha$ w.r.t. $\equiv_{\Delta}^1$ will be denoted by $|\alpha|_\Delta$.  By adapting the proof of  Proposition~\ref{Adelta} it is easy to prove the following:

\begin{proposition}
The structure $\mathcal{A}_{\Delta} \defin\langle A_{\Delta}, \bar\wedge, \bar\vee, \bar\to, 0_{\Delta},1_{\Delta}\rangle$ is a Boolean algebra with the following operations: $|\alpha|_\Delta \bar\# |\beta|_\Delta \defin |\alpha \# \beta|_\Delta$ for any $\# \in \{\wedge, \vee, \to\}$, $0_{\Delta} \defin |\varphi \wedge (\neg \varphi \wedge \circ \varphi)|_{\Delta}$ and  $1_{\Delta} \defin |\varphi \vee \neg \varphi|_{\Delta}$, for any sentence~$\varphi$.
\end{proposition}

\noindent It is worth noting that $\bar{\sneg}|\alpha|_\Delta \defin |\sneg\alpha|_\Delta$ is the Boolean complement of  $|\alpha|_\Delta$ in $\mathcal{A}_{\Delta}$.\footnote{Observe that the occurrence of \sneg\ on the right-hand side of  the definition denotes a classical negation definable in $\lfium_\circ$, as explained in Remark~\ref{univ-Henk}.}
As it was done with \qmbc, the construction of the canonical model for $\qlfium_\circ$ w.r.t. $\Delta$ requires the use of the MacNeille-Tarski completion of the Boolean algebra $\mathcal{A}_{\Delta}$. Then, consider the following:

\begin{definition}\label{nma1}
 Let $(C\mA_{\Delta}, \ast)$ be the MacNeille-Tarski completion of $\mathcal{A}_{\Delta}$. The twist structure for $\lfium_\cons$ over $C\mA_{\Delta}$ will be denoted by  $\mathcal{T}({\Delta})$, and its domain will be denoted by $T(\Delta)$. The associated logical matrix will be denoted by $\mathcal{MT}({\Delta}) \defin (\mathcal{T}({\Delta}), D_{\Delta})$. 
\end{definition}

\begin{remark}\label{rem1}
It is worth noting that $(\ast(|\alpha|_{\Delta}), \ast(|\beta|_{\Delta})) \in D_{\Delta} \ \mbox{ iff } \ \Delta \vdash_{\qlfium_\circ} \alpha$.
\end{remark}

\begin{definition} (Canonical Structure) \label{str1}
Let $\Theta$ be a signature with some individual constant. Let $\Delta \subseteq Sen(\Theta)$ be non-trivial in $\qlfium_\circ$, let $\mathcal{MT}({\Delta})$ be the matrix as in Definition~\ref{nma1}, and let $U \defin $~\ctert. The {\em canonical $\qlfium_\circ$-structure induced by $\Delta$} is the  structure  $\mathfrak{A}_\Delta = \langle U, I_{\mathfrak{A}_\Delta} \rangle$ over $\mathcal{MT}({\Delta})$ and $\Theta$ such that:
\begin{itemize}
\item[-] $c^{\mathfrak{A}_\Delta} = c $, for every individual constant $c \in \mathcal{C}$;
\item[-] $f^{\mathfrak{A}_\Delta}: U^n \to U$ is such that $f^{\mathfrak{A}_\Delta} (t_1, \ldots, t_n) = f(t_1, \ldots, t_n)$, for every function symbol $f$ of arity $n$;
\item[-] $P^{\mathfrak{A}_\Delta}(t_1, \ldots, t_n) = (\ast(|\varphi|_{\Delta}), \ast(|\neg\varphi|_{\Delta})) $ with $\varphi = P(t_1, \ldots, t_n)$, for every predicate symbol $P$ of arity $n$.
\end{itemize}
\end{definition}

\noindent Note that $|\varphi|_{\Delta} \bar\vee \,|\neg \varphi|_{\Delta} = |\varphi \vee \neg \varphi|_{\Delta}= 1$ and so $\ast(|\varphi|_{\Delta}) \,\vee \,\ast(|\neg \varphi|_{\Delta}) = \ast(|\varphi \vee \neg \varphi|_{\Delta}) = 1$. Hence, $P^{\mathfrak{A}_\Delta}(t_1, \ldots, t_n) \in T({\Delta})$ and so  $\mathfrak{A}_\Delta $ is indeed a structure over $\mathcal{MT}({\Delta})$ and $\Theta$.

Let $(\cdot)^\triangleright:(Ter(\Theta_U) \cup For(\Theta_U)) \to (Ter(\Theta) \cup For(\Theta))$ be the function introduced in  Definition~\ref{transla} such that $\left(\,s\,\right)^\triangleright$ is  obtained from $s$ by substituting every occurrence of a constant $\bar{t}$ by the term  $t$ itself. Clearly $(t)^\triangleright = \termvalue{t}^{\widehat{\mathfrak{A}_{\Delta}}}$ for every $t \in CTer(\Theta_{U})$. By adapting the proof of Lemma~\ref{quantOK} it follows:

\begin{lemma} \label{quantOK1}
Let $\Delta \subseteq$ \sent\  be a set of sentences over a signature $\Theta$ such  that $\Delta$ is a $C$-Henkin theory in $\qlfium_\circ$ for a nonempty set $C$ of individual constants of $\Theta$, and $\Delta$ is maximally non-trivial with respect to $\varphi$ in $\qlfium_\circ$, for some sentence $\varphi$.
Then, for every formula $\psi(x)$ in which $x$ is the unique variable (possibly) occurring free, it holds:\\[1mm]
(1) $|\forall x \psi|_\Delta = \bigwedge_{\mA_\Delta} \{ |\psi[x/t]|_\Delta \ :  \ t \in CTer(\Theta)\}$, where $\bigwedge_{\mA_\Delta}$ denotes  an existing infimum in the Boolean algebra $\mathcal{A}_{\Delta}$;\\[1mm]
(2) $|\exists x \psi|_\Delta = \bigvee_{\mA_\Delta} \{ |\psi[x/t]|_\Delta \ :  \ t \in CTer(\Theta)\}$, where $\bigvee_{\mA_\Delta}$ denotes  an existing supremum in the Boolean algebra $\mathcal{A}_{\Delta}$.
\end{lemma}

\begin{proposition}\label{rem2}
Let $\Delta \subseteq$ \sent\ be as in Lemma~\ref{quantOK1}. Then, the interpretation map $\termvalue{\cdot}^{\mathfrak{A}_{\Delta}}:Sen(\Theta_U) \to T({\Delta})$ is such that   $\termvalue{\psi}^{\mathfrak{A}_{\Delta}} = (\ast(|(\psi)^\triangleright|_{\Delta}), \ast(|(\neg\psi)^\triangleright|_{\Delta}))$ for every sentence $\psi$. 
Moreover, $\termvalue{\psi}^{\mathfrak{A}_{\Delta}} \in D_{\Delta} \mbox{ iff } \Delta \vdash_{\qlfium_\circ} (\psi)^\triangleright$. In particular, $\termvalue{\psi}^{\mathfrak{A}_{\Delta}} \in D_{\Delta} \mbox{ iff } \Delta \vdash_{\qlfium_\circ} \psi$ for every $\psi \in Sen(\Theta)$. 
\end{proposition}
\begin{proof}
The proof is done by induction on the complexity of the sentence $\psi$ in $Sen(\Theta_U)$. If $\psi=P(t_1,\ldots,t_n)$ is atomic then, by using Definition \ref{val1}, the fact that $\termvalue{t}^{\widehat{\mathfrak{A}_{\Delta}}}=(t)^\triangleright$ for every $t \in CTer(\Theta_{U})$, and Definition~\ref{str1}, we have:\\[1mm]
$\termvalue{\psi}^{\mathfrak{A}_{\Delta}} = P^{\mathfrak{A}_\Delta}(\termvalue{t_1}^{\widehat{\mathfrak{A}_{\Delta}}}, \ldots, \termvalue{t_n}^{\widehat{\mathfrak{A}_{\Delta}}})=P^{\mathfrak{A}_\Delta}((t_1)^\triangleright, \ldots,(t_n)^\triangleright)$\\
\hspace*{1cm} $=(\ast(|(\psi)^\triangleright|_{\Delta}), \ast(|(\neg\psi)^\triangleright|_{\Delta}))$.\\[1mm]
If $\psi = \neg \beta$ then, by Definition~\ref{val1} and by induction hypothesis,
$$\termvalue{\psi}^{\mathfrak{A}_{\Delta}} = \tilde{\neg}\termvalue{\beta}^{\mathfrak{A}_{\Delta}} = \tilde{\neg}(\ast(|(\beta)^\triangleright|_{\Delta}), \ast(|(\neg\beta)^\triangleright|_{\Delta}))=(\ast(|(\neg\beta)^\triangleright|_{\Delta}), \ast(|(\beta)^\triangleright|_{\Delta})).$$
But $|(\beta)^\triangleright|_{\Delta} = |(\neg\neg\beta)^\triangleright|_{\Delta}$, by (\axdneg). Hence,  $\termvalue{\psi}^{\mathfrak{A}_{\Delta}} = (\ast(|(\psi)^\triangleright|_{\Delta}), \ast(|(\neg\psi)^\triangleright|_{\Delta}))$.\\[1mm]
If $\psi = \cons \beta$ then, by Definition~\ref{val1}, by induction hypothesis, the definition of the operations in $\mathcal{A}_{\Delta}$ and the fact that $\ast$ is a homomorphism of Boolean algebras, \\[1mm]
$\termvalue{\psi}^{\mathfrak{A}_{\Delta}} = \tilde{\cons}\termvalue{\beta}^{\mathfrak{A}_{\Delta}} = \tilde{\cons}(\ast(|(\beta)^\triangleright|_{\Delta}), \ast(|(\neg\beta)^\triangleright|_{\Delta}))$\\
\hspace*{1cm} $=(\ast(|(\sneg(\beta \wedge \neg\beta))^\triangleright|_{\Delta}), \ast(|(\beta \wedge \neg\beta)^\triangleright|_{\Delta}))$.\\[1mm]
But $|(\sneg(\beta \wedge \neg\beta))^\triangleright|_{\Delta} = |(\cons\beta)^\triangleright|_{\Delta}$ and  $|(\beta \wedge \neg\beta)^\triangleright|_{\Delta} = |(\neg\cons\beta)^\triangleright|_{\Delta}$,  by Remark~\ref{cons-defin}. Hence, $\termvalue{\psi}^{\mathfrak{A}_{\Delta}} = (\ast(|(\psi)^\triangleright|_{\Delta}), \ast(|(\neg\psi)^\triangleright|_{\Delta}))$.\\[1mm]
If $\psi = \alpha \# \beta$ for $\# \in \{\wedge, \vee, \to\}$, the proof is analogous, but now axioms (\axnegou), (\axnege) and (\axnegimp) are required.\\[1mm]
If $\psi = \forall x \beta$ then, by Lemma~\ref{quantOK1} and using that $U=CTer(\Theta)$,
$|\forall x \beta|_\Delta = \bigwedge_{\mA_\Delta} \{ |\beta[x/t]|_\Delta \ :  \ t \in U\}$ and so $\ast(|\forall x \beta|_\Delta) = \bigwedge_{C\mA_{\Delta}} \{ \ast(|\beta[x/t]|_\Delta) \ :  \ t \in U\}$. Analogously, $\ast(|\exists x \beta|_\Delta) = \bigvee_{C\mA_{\Delta}} \{ \ast(|\beta[x/t]|_\Delta) \ :  \ t \in U\}$. Then, by Definition~\ref{val1}, by induction hypothesis and by axiom~${\bf (Ax\neg\forall)}$: \\[2mm]
\indent
$\begin{array}{lll}
\termvalue{\forall x \beta}^{\mathfrak{A}_{\Delta}} &=& \big(\bigwedge_{t \in U} (\termvalue{\beta[x/\bar t]}^{\mathfrak{A}_{\Delta}})_1, \, \bigvee_{t \in U} (\termvalue{\beta[x/\bar t]}^{\mathfrak{A}_{\Delta}})_2 \big)\\[2mm]
&=& \big(\bigwedge_{t \in U} \ast(|(\beta[x/\bar t])^\triangleright|_\Delta), \, \bigvee_{t \in U} \ast(|(\neg\beta[x/\bar t])^\triangleright|_\Delta) \big)\\[2mm]
&=&\big({\ast}(|(\forall x\beta)^\triangleright|_{\Delta}), \, {\ast}(|(\exists x\neg\beta)^\triangleright|_{\Delta})\big)= \big({\ast}(|(\forall x \beta)^\triangleright|_{\Delta}), \, {\ast}(|(\neg\forall x \beta)^\triangleright|_{\Delta})\big).
\end{array}
$

\

\noindent Hence, $\termvalue{\psi}^{\mathfrak{A}_{\Delta}} = (\ast(|(\psi)^\triangleright|_{\Delta}), \ast(|(\neg\psi)^\triangleright|_{\Delta}))$. \\[1mm]
If $\psi = \exists x \beta$, the proof is analogous to the previous case. 

This shows that $\termvalue{\psi}^{\mathfrak{A}_{\Delta}} = (\ast(|(\psi)^\triangleright|_{\Delta}), \ast(|(\neg\psi)^\triangleright|_{\Delta}))$ for every sentence $\psi$. The rest of the proof follows by Remark~\ref{rem1}.
\end{proof} 

\begin{theorem} {\em (Completeness  of $\qlfium_\circ$ restricted to sentences w.r.t. first-order twist structures)} 
If $\Gamma \models_{\qlfium_\circ} \varphi$ then $\Gamma \vdash_{\qlfium_\circ} \varphi$, for every $\Gamma \cup \{\varphi\} \subseteq$ \sent. 
\end{theorem}
\begin{proof}
Suppose that  $\Gamma \cup \{\varphi\} \subseteq$ \sent\ is such that $\Gamma \nvdash_{\qlfium_\circ} \varphi$. By adapting  Theorem~7.5.3 in~\cite{CC16} to $\qlfium_\circ$,\footnote{As observed in the proof of Theorem~\ref{comp-sent-Qmbc-swap} above, Theorem~7.5.3 in~\cite{CC16} also holds for the definition of $C$-Henkin theory adopted in this paper.} there exists a $C$-Henkin theory $\Delta^{H}$ over $\Theta_{C}$ in $\qlfium_\circ$ for some nonempty set $C$ of new individual constants such that $\Gamma \subseteq \Delta^{H}$ and, in addition:  $\Gamma \vdash_{\qlfium_\circ} \alpha$ iff $\Delta^{H} \vdash^{C}_{\qlfium_\circ} \alpha$, for every $\alpha \in $ \sent. In consequence, $\Delta^{H} \nvdash^{C}_{\qlfium_\circ} \varphi$. Because of Proposition~\ref{saturated1}, there is a set of sentences $\overline{\Delta^{H}}$ in $\Theta_{C}$ containing $\Delta^{H}$ which is maximally non-trivial with respect to $\varphi$  in $\qlfium_\circ$ (restricted to $Sen(\Theta_C)$), and such that $\overline{\Delta^{H}}$ is also a $C$-Henkin theory over $\Theta_{C}$ in $\qlfium_\circ$. Consider now $\mathcal{MT}(\overline{\Delta^{H}})$ and $\mathfrak{A}_{\overline{\Delta^{H}}}$ as in Definitions~\ref{nma1} and~\ref{str1},  respectively. By Proposition~\ref{rem2}, $\termvalue{\alpha}^{\mathfrak{A}_{\overline{\Delta^{H}}}} \in D_{\overline{\Delta^{H}}} \mbox{ iff } \overline{\Delta^{H}} \vdash^{C}_{\qlfium_\circ} \alpha$, for every $\alpha$ in $Sen(\Theta_{C})$. But then $\termvalue{\gamma}^{\mathfrak{A}_{\overline{\Delta^{H}}}} \in D_{\overline{\Delta^{H}}}$ for every $\gamma \in \Gamma$ and $\termvalue{\varphi}^{\mathfrak{A}_{\overline{\Delta^{H}}}} \notin D_{\overline{\Delta^{H}}}$. Now, let $\mathfrak{A}$ the reduct of $\mathfrak{A}_{\overline{\Delta^{H}}}$ to $\Theta$. Hence, $\mathfrak{A}$ is a structure over $\mathcal{MT}(\overline{\Delta^{H}})$ and $\Theta$ such that $\termvalue{\gamma}^{\mathfrak{A}} \in D_{\overline{\Delta^{H}}}$ for every $\gamma \in \Gamma$ but $\termvalue{\varphi}^{\mathfrak{A}} \notin D_{\overline{\Delta^{H}}}$. This means that $\Gamma \not\models_{\qlfium_\circ} \varphi$.
\end{proof}

\

\begin{corollary} {\em (Completeness of \qmbc\ w.r.t. first-order twist structures)} \label{comp-QLFI1} Let $\Gamma \cup \{\varphi\} \subseteq For(\Theta)$.  If $\Gamma \models_{\qlfium_\circ} \varphi$ then $\Gamma\vdash_{\qlfium_\circ} \varphi$.
\end{corollary}

\section{Completeness of $\qlfium_\circ$ w.r.t. structures over $\matM_{LFI1}$} \label{comp3val}

In Remark~\ref{twistA2} it was observed that $\mathcal{T}_{\mA_2}$, the twist structure for $\lfium_\cons$ defined over the two-element Boolean algebra $\mA_2$, coincides (up to names) with the 3-valued algebra \aptz\ underlying the matrix $\matM_{LFI1}$ and, moreover, $\mathcal{MT}_{\mA_2}$ coincides with the 3-valued characteristic matrix $\matM_{LFI1}$ of \lfiuml. 
Recall that $1$, $\frac{1}{2}$ and $0$ are identified with $(1,0)$, $(1,1)$ and $(0,1)$, respectively.  Let  $\mathfrak{A}$ be  a $\qlfium_\circ$-structure over $\mA_2$. If $\varphi$ is a  formula in which $x$ is the unique variable (possibly) occurring free, let $X= \{\termvalue{\varphi[x/\bar{a}]}^{\mathfrak{A}} : a \in U\}$. Then:
$$\termvalue{\forall x \varphi}^{\mathfrak{A}} =   \begin{cases} 1 \mbox{ if } X = \{1\} \\[1mm] 
\frac{1}{2} \mbox{ if }  X \subseteq \{1,\frac{1}{2}\} \mbox{ and }  \frac{1}{2} \in X \\[1mm] 
0 \mbox{ if }  0 \in X  \end{cases} \,
\termvalue{\exists x \varphi}^{\mathfrak{A}} =   \begin{cases} 1 \mbox{ if } 1 \in X \\[1mm] 
\frac{1}{2} \mbox{ if }  X \cap \{1,\frac{1}{2}\} = \{\frac{1}{2}\} \\[1mm] 
0 \mbox{ if }  X = \{0\}  \end{cases}$$

\noindent 
In Section~\ref{compM5} it was obtained a characterization of \qmbc\ in terms of swap structures over the 5-element characteristic Nmatrix of \mbc, which coincides with the one given in~\cite{avr:zam:07}. That result can be easily  adapted to $\qlfium_\circ$,  by proving that $\qlfium_\circ$ can be characterized by first-order structures defined  over $\matM_{LFI1}$. Indeed, it is possible to adapt Theorem~\ref{val-bival-qmbC} to $\qlfium_\circ$, taking into account that the bivaluations for $\qlfium_\cons$ satisfy aditional clauses, see~\cite[Definition~7.9.16]{CC16}. This lead us to the following result, in view of  Remark~\ref{twistA2}  (details of the proof will be omitted):

\begin{theorem} {\em (Adequacy of $\qlfium_\circ$ w.r.t. first-order structures over $\matM_{LFI1}$)} \label{adeq-LFI1-3val}
For every set $\Gamma \cup \{\varphi\} \subseteq For(\Theta)$:  $\Gamma\vdash_{\qlfium_\circ} \varphi$ \ iff \ $\Gamma\models_{(\mathfrak{A}, \matM_{LFI1})} \varphi$ for every  structure $\mathfrak{A}$ over $\Theta$ and $\matM_{LFI1}$.
\end{theorem}

\noindent
The latter result is a variant (up to language) of the  adequacy theorem of first-order \dacdot\ w.r.t. first-order structures given in~\cite{dot:85a} (see also~\cite{dot:82,dot:85b,dot:87}). Indeed, the semantics in terms of first-order structures over $\matM_{LFI1}$ is equivalent to the 3-valued first-order structures proposed by D'Ottaviano in~\cite{dot:85a} for a quantified version of \dacdot, given that $\lfium_\circ$ is equivalent, up to language, to \dacdot. This shows that the twist-structures semantics for $\qlfium_\circ$ constitutes  a generalization, to any complete Boolean algebra, of the above mentioned semantics for first-order \dacdot. 

The extension of $\qlfium_\circ$ with standard equality is straigtforward, taking into account the construction for \qmbceq\ presented in Section~\ref{secEq}. It is worth noting that, when restricted to structures over $\matM_{LFI1}$, there are differences with D'Ottaviano's approach to first-order \dacdot\ with equality. Indeed, she assumes that the equality  must be classical, that is, every formula $\cons(t_1\approx t_2)$ is valid in her system, contrary to what happens in $\qlfium_\circ$ with equality.  

\section{Final remarks}

In this paper, the semantical frameworks for \qmbc\ already proposed in the literature were extended to a vast class of models based on the non-deterministic algebras known as swap structures. Indeed, the Nmatrix semantics proposed in~\cite{avr:zam:07} and the  semantics given by  interpretations (i.e., standard Tarskian structures plus bivaluations) considered in~\cite{car:etal:14} and~\cite{CC16} coincide, and are particular cases of the swap structures semantics introduced here, as it was shown along this paper.

The advantage of considering models based on  a class of swap structures instead of `classical' models based on a finite Nmatrix (as done in~\cite{avr:zam:07}) is that this enlarged class of models allows us to consider applications  to another fields such, for instance, algebraic logic (as done in~\cite{CFG18}) or paraconsistent set theory. Concerning the latter, the Boolean valued models for set theory could be generalized to this setting, obtaining so swap structures models for several paraconsistent set theories based on \lfis, along the lines of the twist-valued models introduced in~\cite{CC19}.

Two important model-theoretic results for \qmbc\ (and some of its axiomatic extensions) were obtained by Ferguson in~\cite{ferg:18}: \L o\'s' ultraproducts theorem, and a suitable version of the Keisler-Shelah isomorphism theorem, which states that two \qmbc-models are strongly elementarily equivalent iff there exists an ultrafilter $\mathcal{U}$  such that the corresponding ultrapowers over $\mathcal{U}$ are strongly isomorphic. The notions of strong elementary equivalence and strong isomorphism were introduced in~\cite{ferg:18}, as well as an adaptation of the method of atomization introduced by Skolem, which was used in order to prove the Keisler-Shelah theorem for quantified \lfis. It would be interesting to adapt Ferguson's notions and constructions to the present semantical framework for quantified \lfis. 

In other line of research, it would be intersting to extend the techniques developed in~\cite{avr:kon:zam:13} for generating cut-free Gentzen-type calculi for propositional \lfis\ from Nmatrix semantics to the first-order framework described here.

\

\noindent{\bf Acknowledgements:}
This paper is the full version of the  extended abstract~\cite{CFG19a}.
Coniglio was financially supported by an individual research
grant from CNPq, Brazil (308524/2014-4). Figallo-Orellano acknowledges financial support from a post-doctoral grant  from FAPESP, Brazil (2016/21928-0). Golzio was financially supported by a post-doctoral grant  from CNPq, Brazil  (150064/2018-7) and by  a post-doctoral grant  from FAPESP, Brazil (2019/08442-9).

\end{document}